\documentclass[12pt,reqno]{amsart}
\usepackage{enumerate}
\usepackage{amsfonts}
\usepackage{amsmath}
\usepackage{amsthm}
\usepackage{latexsym}
\usepackage{multicol}
\usepackage{verbatim}
\usepackage{amscd}
\usepackage{amssymb}
\usepackage{ytableau}
\usepackage{amsmath, amscd,wasysym,xcolor}
\usepackage{graphicx}
\usepackage{setspace}
\usepackage{tikz}
\graphicspath{ {images/} }
\usepackage{multicol}
\usepackage[colorlinks=true, linkcolor=blue, citecolor=red, urlcolor=blue]{hyperref}
\usepackage{hyperref}
\usepackage{hyperref}
\advance\textwidth by 1.2in \advance\oddsidemargin by -.65in \advance\evensidemargin by -.65in
\newtheorem{cor}{Corollary}
\newtheorem{lem}{Lemma}
\newtheorem{prop}{Proposition}

\theoremstyle{definition}
\newtheorem{defn}{Definition}

\theoremstyle{definition}
\newtheorem{thm}{Theorem}

\newtheorem{rem}{Remark}

\newtheorem{q}{Example}

\newcounter{cnt}
 \makeatletter
\def\mydggeometry{\makeatletter\dg@YGRID=1\dg@XGRID=20\unitlength=0.003pt\makeatother}
\makeatother \theoremstyle{remark}

\numberwithin{equation}{section}

\let\bwdg\bigwedge
\def\bigwedge{{\textstyle\bwdg}}

\newcommand{\thmref}[1]{Theorem~\ref{#1}}

\newcommand{\secref}[1]{Section~\ref{#1}}
\newcommand{\lemref}[1]{Lemma~\ref{#1}}
\newcommand{\propref}[1]{Proposition~\ref{#1}}

\newcommand{\remref}[1]{Remark~\ref{#1}}

\newcommand{\Ker}{\operatorname{Ker}}

\newcommand{\nc}{\newcommand}
\newcommand{\rnc}{\renewcommand}

\nc{\cal}{\mathcal} \nc{\goth}{\mathfrak} \rnc{\bold}{\mathbf}
\nc{\fk}{\mathfrak}
\nc{\germ}{\mathfrak}

\renewcommand{\Bbb}{\mathbb}
\nc\bomega{{\mbox{\boldmath $\omega$}}}
\nc\bxi{{\mbox{\boldmath $\xi$}}}
\nc\blambda{{\mbox{\boldmath $\lambda$}}}
\nc\bOmega{{\mbox{\boldmath $\Omega$}}}
\nc\bchi{{\mbox{\boldmath $\chi$}}}\nc\bnu{{\mbox{\boldmath $\nu$}}}
\nc\bpsi{{\mbox{\boldmath $\Psi$}}} \nc\bast{{\mbox{\boldmath $\ast$}}}
\nc\balpha{{\mbox{\boldmath $\alpha$}}}
\nc\bpi{{\mbox{\boldmath $\pi$}}}
\nc\bsigma{{\mbox{\boldmath $\sigma$}}} \nc\bcN{{\mbox{\boldmath $\cal{N}$}}} \nc\bcm{{\mbox{\boldmath $\cal{M}$}}} 
\nc\bLambda{{\mbox{\boldmath$\Lambda$}}} 
\nc\bll{{\mbox{\boldmath$\ell$}}} \nc\bgamma{{\mbox{\boldmath$\gamma$}}}

\nc\p{{\mbox{\boldmath$\rho$}}} 
\nc\bmu{{\mbox{\boldmath$\mu$}}}

\newcommand{\mfg }{\mathfrak{g}}

\makeatletter
\def\subsection{\def\@secnumfont{\bfseries}\@startsection{subsection}{2}%
{\parindent}{.5\linespacing\@plus.5\linespacing}{-.5em}%
{\normalfont\bfseries}}
\makeatother

\nc{\Hom}{\operatorname{Hom}}
\nc{\td}{\operatorname{\tilde{d}}}\nc{\D}{\operatorname{d}}
\nc{\mode}{\operatorname{mod}}
\nc{\End}{\operatorname{End}} \nc{\wh}[1]{\widehat{#1}} \nc{\Ext}{\operatorname{Ext}} \nc{\ch}{\text{ch}} \nc{\ev}{\operatorname{ev}}
\nc{\Ob}{\operatorname{Ob}} \nc{\soc}{\operatorname{soc}} \nc{\rad}{\operatorname{rad}} \nc{\head}{\operatorname{head}}
\nc{\A}{\operatorname{\bold{b}}}

\def\ker{\operatorname{Ker}}

\def\ch{\operatorname{ch}}

\nc{\Cal}{\cal} \nc{\Xp}[1]{X^+(#1)} \nc{\Xm}[1]{X^-(#1)}
\nc{\on}{\operatorname} \nc{\Z}{{\mathbb{Z}}} \nc{\J}{{\cal J}} \nc{\C}{{\mathbb{C}}} \nc{\Q}{{\bold Q}} \nc{\R}{{\mathbb{R}}}
\nc{\K}{\bold{\kappa}}

\nc{\F}{\operatorname{\boF}} 

\nc{\N}{{\Bbb N}} \nc\boa{\bold a} \nc\bob{\bold b} \nc\boc{\bold c} \nc\bod{\bold d} \nc\boe{\bold e} \nc\bof{\bold f} \nc\bog{\bold g}
\nc\boh{\bold h} \nc\boi{\bold i} \nc\boj{\bold j} \nc\bok{\bold k} \nc
\bol{\bold l} \nc\bom{\bold m} \nc\bon{\bold n} \nc\boo{\bold o}
\nc\bop{\bold p} \nc\boq{\bold q} \nc\bor{\bold r} \nc\bos{\bold s} \nc\boT{\bold t} \nc\boF{\bold F} \nc\bou{\bold u} \nc\bov{\bold v}
\nc\bow{\bold w} \nc\boz{\bold z} \nc\boy{\bold y} \nc\ba{\bold A} \nc\bb{\bold B} \nc\bc{\bold C} \nc\bd{\bold D} \nc\be{\bold E} \nc\bg{\bold
G} \nc\bh{\bold H} \nc\bi{\bold I} \nc\bj{\bold J} \nc\bk{\bold K} \nc\bl{\bold L} \nc\bm{\bold M} \nc\bn{\bold N} \nc\bo{\bold O} \nc\bp{\bold
P} \nc\bq{\bold Q} \nc\br{\bold R} \nc\bs{\bold S} \nc\bt{\bold T} \nc\bu{\bold U} \nc\bv{\bold V} \nc\bw{\bold W} \nc\bz{\bold Z} \nc\bx{\bold
x} \nc\KR{\bold{KR}} \nc\rk{\bold{rk}} \nc\het{\text{ht }}

\nc\fn{{fin}}  \nc\af{{aff}}  \nc\tr{{tor}} \nc\btilde{\bold{\tilde{\bold{H}}}}

\nc{\mpp}{\rotatebox[origin=c]{180}{\pm}}
\nc{\iA}{\rotatebox[origin=c]{180}{A}}
\nc{\iB}{\rotatebox[origin=c]{180}{B}}
\nc{\iC}{\rotatebox[origin=c]{180}{C}}

\nc\eps{\epsilon}
\nc\toa{\tilde a} \nc\tob{\tilde b} \nc\toc{\tilde c} \nc\tod{\tilde d} \nc\toe{\tilde e} \nc\tof{\tilde f} \nc\tog{\tilde g} \nc\toh{\tilde h}
\nc\toi{\tilde i} \nc\toj{\tilde j} \nc\tok{\tilde k} \nc\tol{\tilde l} \nc\tom{\tilde m}  \nc\ton{\tilde n} \nc\too{\tilde o} \nc\toq{\tilde q}
\nc\tor{\tilde r} \nc\tos{\tilde s} \nc\toT{\tilde t} \nc\tou{\tilde u} \nc\tov{\tilde v} \nc\tow{\tilde w} \nc\toz{\tilde z}
\begin{document}

\title{Graded representations of current Lie superalgebras $\mathfrak{sl}(1|2)[t]$
 }

\author{Shushma Rani}
\author{Divya Setia}

 \address{Department of Mathematics, Indian Institute of Science, Bangalore, 560012, India.}
 \email{shushmarani@iisc.ac.in, shushmarani95@gmail.com}	
	\address{Indian Institute of Science Education and Research Berhampur, Govt. ITI Building, Engineering School Rd, Junction, Gajapati Nagar, Brahmapur, 760010, Odisha, India}
	\email{divyasetia01@gmail.com}
	
\date{}
\thanks{}

\subjclass [2020]{16S30,17B05, 17B10 17B35 17B65, 17B67, 17B70, 05E10}
\keywords{Super POP, CV-modules, Local Weyl module, Generalized Kac modules, Current Lie superalgebras}


\begin{abstract}
This paper is the study of finite-dimensional graded representations of current lie superalgebras $\mathfrak{sl}(1|2)[t]$. We define the notion of super POPs, a combinatorial tool to provide another parametrization of the basis of the local Weyl module given in \cite{macedo}. We derive the graded character formula of local Weyl module for $\mathfrak{sl}(1|2)[t]$. Furthermore, we construct a short exact sequence of Chari-Venkatesh modules for $\mathfrak{sl}(1|2)[t]$. As a consequence, we prove that Chari-Venkatesh modules are isomorphic to the fusion of generalized Kac modules.
\end{abstract}

\maketitle


\section{Introduction}
Let $\mathfrak{a}$ be a simple finite-dimensional Lie algebra over the complex field $\mathbb{C}$ and associated to $\mathfrak{a}$, define the current Lie algebra $\mathfrak{a}[t]: = \mathfrak{a} \otimes \mathbb{C}[t]$ as a vector space where $\mathbb{C}[t]$ denotes the polynomial ring in indeterminate $t$. The representation theory of current Lie algebras has gained significant attention in recent years. In \cite{chari-pressley}, authors introduced the notion of local Weyl modules (denoted by $W_{loc}(\lambda)$) that are universal objects in the category of finite-dimensional graded highest weight modules of $\mathfrak{a}[t]$ where $\lambda$ is a dominant integral weight of $\mathfrak{a}$. In \cite{CI15}, it was proved that the graded character of the local Weyl module coincides with the non-symmetric Macdonald polynomial specialized at $t=0$. This result establishes a strong connection between graded representations of current Lie algebras with algebraic combinatorics.

In \cite{chari-loktev}, authors computed an explicit basis known as Chari-Loktev basis of the local Weyl module for the current Lie algebra of type $A$. Another parametrization of the Chari-Loktev basis of the local Weyl module was given in \cite{pop} where the authors introduced the notion of POPs (partition overlay patterns) and proved that POPs with bounded sequence $\lambda$ parametrize the Chari-Loktev basis of $W_{loc}(\lambda)$. In \cite{cv}, an important family of finite-dimensional quotients of local Weyl modules was defined, known as Chari-Venkatesh modules (in short CV modules).

In \cite{FL99}, a family of $\mathbb{Z}$-graded modules for current Lie algebras called fusion product modules were introduced where the underlying vector space comprises tensor product of highest weight $\mathfrak{a}[t]$-modules. Also, the definition of fusion product module is based on certain set of evaluation parameter, but it was conjectured that the structure of these modules remains unchanged regardless of the chosen set of evaluation parameter. This conjecture was confirmed in several cases \cite{chari-pressley, chari-loktev, cv, FL99, FL07, Naoi17, ravinder14, TS23, SRK24} over the past two decades. CV modules is an important tool in confirming this conjecture for various cases \cite{cv,Naoi17,ravinder14,TS23,SRK24}.

Let us move our attention to the representation theory of Lie superalgebras. The representation theory of Lie superalgebras has been extensively studied over the past few decades. Let $\mathfrak{g}$ denote the Lie superalgebra and the current Lie superalgebra associated to $\mathfrak{g}$ is denoted by $\mathfrak{g}[t]$, a class of infinte-dimensional Lie superalgebras, which carry a $\mathbb Z_+$-grading induced by the grading of the polynomial algebra $\mathbb C[t]$. Several authors have attempted to construct results for current Lie superalgebras analogous to the one established for current Lie algebras. 

In \cite{savage,bagci}, the notion of local Weyl modules for current Lie superalgebras was introduced. Then the dimension, basis, and graded character formula of a local Weyl module were also derived in \cite{Kus18, FM17} for the current Lie superalgebra $\mathfrak{osp}(1|2)[t]$. It was also proved that the graded character of the local Weyl module coincides with the non-symmetric Macdonald polynomials specialised at $t=0$.

In this paper, we are mainly interested in the representation theory of the current Lie superalgebra $\mathfrak{sl}(1|2)[t]$. The category of finite-dimensional representations of the Lie superalgebra $\mathfrak{osp}(1|2)$ is semisimple. In contrast, the corresponding category for $\mathfrak{sl}(1|2)$ is not semisimple. Therefore, finite-dimensional graded representations of the current Lie superalgebra $\mathfrak{osp}(1|2)[t]$ behave in a similar way as the representation theory of current Lie algebras but it is difficult to explore finite-dimensional graded representations of current Lie superalgebras $\mathfrak{sl}(1|2)[t]$. 

We define the notion of super POPs for the current Lie superalgebra $\mathfrak{sl}(1|2)[t]$ generalizing the notion of POPs given in \cite{pop} for current Lie algebras of type $A$. Super POPs helped us to derive another parametrization of the basis of local Weyl modules given in \cite{macedo} for the current Lie superalgebra $\mathfrak{sl}(1|2)[t]$. In \cite{macedo}, authors computed the dimension and $\mathfrak{g}$-character formula of the local Weyl module in the case of $\mathfrak{sl}(1|2)[t]$ but the graded character formula was still unexplored. Therefore, we compute the graded character formula of the local Weyl module in two different ways, one from the basis given in \cite{macedo} and other by using super POPs. We expect that the graded character formula of the local Weyl module of $\mathfrak{sl}(1|2)[t]$ will coincides with some non-symmetric Macdonald polynomial at $t=0$ similar to the case of $\mathfrak{osp}(1|2)[t]$. 

We develop a short exact sequence of the CV module (denoted by $V(\bxi)$) defined in \cite{macedo} to prove that the fusion product of generalized Kac modules is independent of parameters. This result was also proved in \cite{macedo} by constructing a basis of the CV module for $\mathfrak{sl}(1|2)[t]$. The short exact sequence of CV modules was proved to be an important tool to derive some useful results in the representation theory of current Lie algebras, such as establishing Demazure filtrations of a CV module and proving the Feigin-Loktev conjecture in various cases. This motivates us to define a short exact sequence of CV modules for $\mathfrak{sl}(1|2)[t]$. We expect that the short exact sequence of CV modules will lead to the existence of Demazure filtrations of CV modules. This advancement opens up numerous research directions and provides a framework for tackling many open problems in the literature. 

The structure of this paper is as follows. In \secref{sec2} we give preliminaries, including the definitions of generalized Kac modules and fusion modules. In \secref{sec3}, we recall the definition of local Weyl modules for $\mathfrak{sl}(1|2)[t]$ and introduce the notion of super POPs to establish a bijection between super POPs and the basis of the local Weyl module given in \cite{macedo}. We also compute the graded character formula of the local Weyl module for $\mathfrak{sl}(1|2)[t]$. In \secref{sec4} we focus on CV modules, providing a short exact sequence of CV modules for $\mathfrak{sl}(1|2)[t]$, ultimately showing that CV modules are isomorphic to the fusion of generalized Kac modules.

\section{Preliminaries}\label{sec2}
Throughout this paper, $\mathbb C$ will denote the field of complex numbers, $\mathbb Z_2$ the field of order 2 consisting of $\{0,1\}$ and $\mathbb Z$ the set of integers, $\mathbb Z_+$ (resp. $\mathbb N$)  the set of non-negative integers (resp. positive integers). For $n,r \in \mathbb{Z}_+$, set $${n \choose 0}_q =1, \quad {n \choose r}_q = \frac{(1-q)(1-q^2)\cdots (1-q^n)}{(1-q)\cdots (1-q^r) (1-q) \cdots (1-q^{n-r})} .$$

\subsection{} Let $\mathfrak g=\mathfrak{sl}(1|2)$ be a simple finite-dimensional Lie superalgebra of $3 \times 3$ matrices over $\mathbb C$ of supertrace zero. Since $ \mathfrak g=  \mathfrak g_{\bar 0} \oplus \mathfrak g_{\bar 1}$ with $\mathfrak g_{\bar 1}= \mathfrak g_{-1} \oplus \mathfrak g_{1}$ where $$\begin{array}{rl}
    \mathfrak g_0 =& \left\lbrace \begin{pmatrix}
        a+d & 0 &0\\
        0 & a &b\\
        0 & c & d
    \end{pmatrix}: a,b,c,d \in \mathbb C \right\rbrace  \\
    \mathfrak g_{-1} =& \left\lbrace \begin{pmatrix}
        0 & 0 &0\\
        z & 0 &0\\
        w & 0 & 0
    \end{pmatrix}: z,w \in \mathbb C \right\rbrace  \text{ and } \mathfrak g_{1} = \left\lbrace \begin{pmatrix}
        0 & x & y\\
        0 & 0 &0\\
        0 & 0 & 0
    \end{pmatrix}: x,y \in \mathbb C \right\rbrace
\end{array}$$
with the following Lie bracket $$[A,B] = AB - (-1)^{\bar{i}\bar{j}}BA, \quad \forall A\in \mathfrak g_{\bar i}, B \in \mathfrak g_{\bar j}, \bar i, \bar j \in \mathbb{Z}_{2}.$$
Moreover, the Lie superagebra $\mathfrak{sl}(1|2)$ admits a $\mathbb Z$-gradation $$\mfg = \mfg_{-1}\oplus \mfg_0 \oplus \mfg_{1}$$
\subsection{} Fix the Cartan subalgebra $\mathfrak h \subseteq \mfg$, where the basis of $\mathfrak h$ consists of $\{h_1 : = E_{1,1}+E_{2,2}, h_2 : = E_{2,2}-E_{3,3}\}$. Set $h_3 = E_{1,1}+E_{3,3}$ and we have a root space decomposition of $\mfg = \bigoplus_{\alpha \in \mathfrak h^*}\mfg^{\alpha}$, where $\mfg^{\alpha}:= \{x\in \mfg |[h,x]=\alpha(h)x, \forall h \in \mathfrak h\}$, by the using the adjoint action of $\mathfrak h$ on $\mfg$. Define the unique linear maps $\delta, \epsilon_1, \epsilon_2$, where $$\delta \begin{pmatrix} a& 0& 0\\
0& b& 0\\ 0& 0& c \end{pmatrix} =a, \quad \epsilon_1 \begin{pmatrix} a& 0 &, 0\\ 0& b& 0\\ 0& 0& c \end{pmatrix} =b, \quad \epsilon_2 \begin{pmatrix} a& 0& 0\\ 0& b& 0\\ 0& 0& c \end{pmatrix} =c.$$
Then the set of roots $\phi : = \{\pm \alpha_1 = \pm (\delta-\epsilon_1),\, \pm \alpha_2 = \pm (\epsilon_1-\epsilon_2),\, \pm \alpha_3 = \pm (\delta-\epsilon_2) \}$. Let us choose the set of simple roots to be $\Pi = \{-\alpha_1, \alpha_3\}$. Then the set of positive roots is $\Phi^+ = \{-\alpha_1,\alpha_2 = \alpha_3 - \alpha_1, \alpha_3\}$. Therefore, we have a triangular decomposition of $\mfg = \mathfrak n^+ \oplus \mathfrak h \oplus \mathfrak n^-$, where $\mathfrak n^{\pm}: = \bigoplus_{\alpha \in \Phi^+} \mfg^{\pm \alpha}$. Let $\mathfrak b = \mathfrak h \oplus \mathfrak n^+$ be the Borel subalgebra of $\mfg$ associated to $\Pi$. 

Let $x_\alpha\in \mathfrak g^\alpha, \, \alpha \in \Phi $. Set $y_{\alpha}:= x_{-\alpha}$ and $x_{\alpha_i}:= x_i$. Note that $\mathfrak n^-:= span_{\mathbb C} \{x_1, y_2, y_3\} $ and $\mathfrak n^+:= span_{\mathbb C} \{y_1, x_2, x_3\},$ where $x_1 = E_{12}, x_2 = E_{23}, x_3 = E_{13}, y_1 = E_{21}, y_2 = E_{32}, y_{3} = E_{31}$.   

\subsection{} For $\lambda \in \mathfrak h^*$, denote $\lambda_1 = \lambda(h_1)$ and $\lambda_2 = \lambda(h_2) $ and we can identify $\lambda$ by $(\lambda_1,\lambda_2) \in \mathbb{C}^{2}$. We are interested in graded representations of $\mfg$. A representation $V$ of $\mfg$ is a $\mathbb{Z}_2$-graded vector space if $V = V_0 \oplus V_1$ such that $$\mfg_i V_j \subseteq V_{i+j}, \forall i,j \in \mathbb{Z}_2.$$ Every finite-dimensional irreducible $\mfg$-module is a highest weight module $L(\lambda)$, for some $\lambda \in \mathfrak h^*$. Let $P^+ = \{\lambda \in \mathfrak h^*| L(\lambda) \text{ is finite-dimensional} \}$.
Using the Proposition 2.5 of \cite{macedo}, we have that $$P^+ = \{\lambda \in \mathfrak h^*| \lambda_2 \in \mathbb N\} \cup \{0\}$$
corresponding to the Borel subalgebra $\mathfrak b$. 

\subsection{} 

For $\lambda \in P^+$, the generalized Kac module associated to $\lambda$ (denoted by $K(\lambda)$), is the cyclic $\mathfrak g$-module generated by homogeneous cyclic even vector $k_\lambda$ with defining relations: $$ \mathfrak n^+ k_\lambda=0, \quad hk_\lambda= \lambda(h)k_\lambda,\forall h \in \mathfrak{h} \quad y_2^{\lambda(h_2)+1} k_\lambda=0.$$
For our purpose, we have given the definition of the generalized Kac module corresponding to the Borel subalgebra $\mathfrak{b}$ but it can be defined for any Borel subalgebra of $\mathfrak{g}$ (See \cite[Definition 4.1]{bagci}). For $\lambda \in P^+$, the dimension of $K(\lambda)$ was computed in \cite{macedo}. 

\begin{prop}\label{dimension of kac}
    For $\mathfrak{g}=\mathfrak{sl}(1|2)$, $\dim K(\lambda)= 4\lambda_2$ for all $\lambda \in P^+ \setminus \{0\}.$
\end{prop}

\subsection{} For a Lie superalgebra $\mfg$, the current Lie superalgebra $\mfg[t]: = \mfg \otimes \mathbb{C}[t]$, where $\mathbb{C}[t]$ denotes the polynomial ring in indeterminate $t$. $\mfg[t]$ is a Lie superalgebra with the following Lie bracket $$[a\otimes t^r, b\otimes t^s] = [a,b]\otimes t^{r+s}, \quad \forall a,b \in \mfg, r,s \in \mathbb{Z}_+$$ where the $\mathbb{Z}_2$-grading is given by $(\mfg \otimes \mathbb{C}[t])_j = \mfg_j \otimes \mathbb{C}[t], \forall j\in \mathbb{Z}_2$. Note that $0=[y_3, y_3]= 2 y_3^2$ and $0=[x_1,x_1] = 2x_1^2$ then we have 
\begin{equation} \label{ysqre equation}
    (y_3 \otimes t^a)^2=0, \quad (x_1 \otimes t^a)^2=0,\quad \forall a \geq 0.
\end{equation}
Let $U(\mathfrak{g}[t])$ denote the universal enveloping algebra corresponding to $\mathfrak{g}[t]$. The natural grading on $\mathbb{C}[t]$ defines a $\mathbb{Z}_{+}$-grading on $U(\mathfrak{g}[t])$. With respect to it, every homogeneous element of $U(\mathfrak{g}[t])$ of degree $k$ is a linear combination of elements of the form $(x_{1} \otimes t^{a_{1}})(x_{2} \otimes t^{a_{2}}) \dots (x_{r} \otimes t^{a_{r}}), \,x_i \in \mathfrak{g}, a_i\in \mathbb Z_+,\, 1\leq r \leq k$ with $a_{1}+a_{2}+ \dots + a_{r}=k$. This grading induces a filtration on any cyclic $\mathfrak{g}[t]$-module.

\subsection{}A graded representation of $\mathfrak g[t]$ is a 
$\mathbb Z_+$-graded vector space 
$V= \underset{r\in \mathbb Z_+}{\bigoplus} V[r]$ such that 
$$ \bu(\mathfrak g[t])[s].V[r]\subseteq V[r+s], 
\quad \forall\, r,s\in \mathbb Z_+.$$ If $U, V$ are two graded representations 
of $\mathfrak g[t]$, we say $\psi: U\rightarrow V$ is a  graded 
$\mathfrak g[t]$-module morphism if  $\psi(U[r])\subseteq V[r]$ for all 
$r\in \mathbb Z_+$. For $s\in \mathbb  Z_+$, let $\tau_s$ be the grade-shifting
operator whose action on a graded representation $V$ of $\mathfrak g[t]$ is 
given by $$\tau_s(V)[k] = V[k+s], \quad  \forall\,  k\in \mathbb Z_+.$$  
Let $q$ be an indeterminate. In a graded $\mathfrak g[t]$-module $V$, given by 
$V= \underset{r\in \mathbb Z_+}{\bigoplus} V[r]$, each graded piece $V[r]$ is a
$\mathfrak g$-module. The graded character of $V$ is then defined as 
$$ \ch_{gr} V = \sum\limits_{r \geq 0} q^r \ch_{\mathfrak g} V[r] .$$

\subsection{} Given any $\mathfrak g$ module $V$ and a fixed $z \in \mathbb C$, let $V^z$ be associated evaluation $\mathfrak g[t]$-module with underlying vector space $V$ and the action is defined as $(x \otimes t^k)v= z^k x.v$ for all $x \in \mathfrak g, \, v \in V$ and $k\in \mathbb Z_+.$ If $V_1, V_2, \cdots ,V_n$ are cyclic $\mathfrak{g}$-modules with cyclic homogeneous even generators $v_1, v_2, \cdots, v_n$ respectively and $z_1, z_2, \cdots, z_n$ are distinct complex numbers then $V_1^{z_1} \otimes \cdots \otimes V_n^{z_n}$ is a cyclic $\mathfrak g[t]$-module with cyclic homogeneous even generator $v_1 \otimes \cdots \otimes v_n$. For each $s\in \mathbb Z_+$, define a filtration on this cyclic module by $$V_1^{z_1} \otimes \cdots \otimes V_n^{z_n}[s]= \bigoplus\limits_{0 \leq r \leq s} U(\mathfrak g[t])[r](v_1 \otimes \cdots \otimes v_n).$$ 
The associated graded $\mathfrak g[t]$ module, $$V_1^{z_1} \otimes \cdots \otimes V_n^{z_n}[0] \oplus \bigoplus\limits_{s \in \mathbb N} \dfrac{V_1^{z_1} \otimes \cdots \otimes V_n^{z_n}[s]}{V_1^{z_1} \otimes \cdots \otimes V_n^{z_n}[s-1]}$$ is called the fusion product of $V_1, \cdots, V_n$ with fusion parameters $z_1, \cdots z_n$ and we denote this by $V_1^{z_1} \ast \cdots \ast V_n^{z_n}.$

\section{A basis for local Weyl modules}\label{sec3}
Fix the Lie superalgebra $\mathfrak{g}= \mathfrak{sl}(1|2)$ and the current Lie superalgebra $\mfg[t] = \mathfrak{sl}(1|2)[t]$ for the rest of the paper. In this section, we describe a new parametrization of the basis of the local Weyl module given in \cite{macedo}.
\subsection{}Let us recall the definition of the local Weyl module for $\mfg[t]$ given in \cite{savage}. 
\begin{defn} \label{loc def}
    Let $\psi \in (\mathfrak{h}[t])^*$ such that $\psi({h_2}) \in \mathbb{Z}_{+}$, define the local Weyl module $W(\psi)$ to be a highest weight $\mfg[t]$-module generated by a vector $w_{\psi}$ with the following defining relations $$\mathfrak{n}^{+}[t] w_{\psi} = 0, \quad hw_{\psi} = \psi(h)w_{\psi}, \, \forall h\in \mathfrak{h}[t], \quad (y_2 \otimes 1)^{\psi(h_2)+1}w_{\psi}=0.$$
\end{defn}
$W(\psi)$ is not necessarily a graded local Weyl module, but its associated graded module (denoted by $Gr_{w_\psi} W(\psi)$) is isomorphic to a graded local Weyl module.
$W(\psi)$ is a finite-dimensional $\mathfrak{g}[t]$-module if $\psi|_{\mathfrak h} \in P^+$ (See \cite{bagci}) and we focus on finite-dimensional graded local Weyl modules. 
\begin{rem} [See {\cite[Theorem 3.8]{macedo}}] Let $\psi \in \mathfrak{h}[t]^\ast$ such that $\psi|_\mathfrak h \in P^+$, if there exists $\mu_1, \mu_2, \ldots, \mu_n \in P^+$ and pairwise distinct $z_1,z_2,\cdots,z_s \in \mathbb{C}$ such that $\psi(h\otimes t^p) = z_1^p \mu_1(h) + \cdots+ z_s^p \mu_s(h), \forall h \in \mathfrak h, p\geq 0$ then $W(\lambda) \cong Gr_{w_\psi} W(\psi)$ where $\lambda := \psi|_\mathfrak h.$
\end{rem}

\subsection{} We can consider $\lambda \in P^+$ as an element of $(\mathfrak{h}[t])^*$ by defining $\lambda(h\otimes t^k) = \delta_{0,k}\lambda(h), \forall h \in \mathfrak h, k\geq 0$. Here, $W(\lambda)$ inherits a grading from $\bu(\mfg[t])$ and it is called graded local Weyl module.
For $u\in \bu(\mfg[t])$ and $k \in \mathbb{Z}_{+}$, define $u^{(k)}: = u^k/k!$. Given $r,s \in \mathbb{Z}_+$, define \begin{equation}
    y_{2}(r,s) = \sum_{(b_p)_{p\in \mathbb Z_+}} (y_2 \otimes 1)^{(b_0)}(y_2 \otimes t)^{(b_1)} \cdots (y_2 \otimes t^s)^{(b_s)} 
\end{equation}
where $p,b_p \in \mathbb{Z}_{+}$ are such that $\sum_{p=0}^s b_p = r,\, \sum_{p=0}^s pb_p = s$. The following result was proved in \cite{garland} and reformulated in the present form in \cite{chari-pressley}.
\begin{lem}
    Given $r,s \in \mathbb{Z}_+$, we have $$(x_2 \otimes t)^{(s)}(y_2 \otimes 1)^{(r+s)} -(-1)^s y_2 (r,s) \in \bu(\mfg[t])(\mathfrak{n}^{+}[t] \oplus t\mathfrak{h}[t])$$
\end{lem}

\subsection{} Let us recall a basis of the local Weyl module $W(\psi)$, $\psi \in (\mathfrak{h}[t])^*$ given in \cite{macedo} (Corollary 3.9).
\begin{prop}\label{local_basis}
    Let $\psi \in (\mathfrak{h}[t])^*$, and assume that there exist $\mu_1,\mu_2,\cdots,\mu_s \in P^+$ and pairwise distinct $z_1,z_2,\cdots,z_s \in \mathbb{C}$ such that $\psi(h\otimes t^p) = z_1^p \mu_1(h) + \cdots+ z_s^p \mu_s(h), \forall h \in \mathfrak h, p\geq 0$.\\
    (a) The elements \begin{equation}\label{basis of local}(y_2 \otimes t^{a_1}) \cdots (y_2 \otimes t^{a_j}) (x_1 \otimes t^{b_1})\cdots (x_1 \otimes t^{b_k}) (y_3 \otimes t^{c_1}) \cdots (y_3 \otimes t^{c_l})w_{\psi},\end{equation} such that $0\leq c_1 < \cdots < c_l \leq \psi(h_2)-1,\, 0 \leq b_1 < \cdots <b_k \leq \psi(h_2)-l-1,\, 0 \leq a_1 \leq \cdots \leq a_j \leq \psi(h_2)-l-k-j$, forms a basis of $W(\psi)$.\\
    (b) The character of $W(\psi)$ is given by $$\ch W(\psi) = e^{(\psi(h_1),0)}\big(e^{(-1,-1)}+e^{(-1,0)}+e^{(0,0)}+e^{(0,1)}\big)^{\psi(h_2)}.$$ \qed
\end{prop} 
The character formula given in part(b) of \propref{local_basis} is a character formula $\mathfrak{g}$, but using part (a) of \propref{local_basis} we will provide a graded character formula for the local Weyl module $W(\psi)$ under the assumptions of \propref{local_basis}. In the following graded character formula, for any weight $e(\mu)$, $\mu= (\mu_1, \mu_2)$ we write the notation $u_1^{\mu_1} u_2^{\mu_2}$ where $u_1, u_2$ are indeterminate.
\begin{prop}\label{graded character formula}
    Let $\psi \in (\mathfrak{h}[t])^*$, and assume that there exist $\mu_1,\mu_2,\cdots,\mu_s \in P^+$ and pairwise distinct $z_1,z_2,\cdots,z_s \in \mathbb{C}$ such that $\psi(h\otimes t^p) = z_1^p \mu_1(h) + \cdots+ z_s^p \mu_s(h),\, \forall h \in \mathfrak h, p\geq 0$ then $$\ch_{gr}(W(\psi)) = \sum_{l=0}^n q^{\frac{l(l-1)}{2}} {n \choose l}_{q}\sum_{k=0}^{n-l} q^{\frac{k(k-1)}{2}} {n-l \choose k}_{q}\sum_{j=0}^{n-k-l} {n-k-l \choose j}_{q}u_1^{\psi(h_1)-l-j} u_2^{n-l-k-2j}$$ where $n=\psi(h_2)$.
\end{prop}
\begin{proof}
  The graded character of the set of vectors given in \eqref{basis of local} with fixed $b_1,b_2, \cdots, b_k$ and $c_1, c_2,\cdots, c_l$ is $\sum_{j=0}^{n-k-l} {n-k-l \choose j}_q u_1^{\psi(h_1)-l-j}u_2^{n-l-k-2j}$. Now the graded character of the space $$(x_1 \otimes t^{b_1})\cdots(x_1 \otimes t^{b_k}) (y_3 \otimes t^{c_1}) \cdots (y_3 \otimes t^{c_l})w_{\psi}, \quad 0\leq c_1 < \cdots < c_l \leq n-1,\, 0 \leq b_1 < \cdots <b_k \leq n-l-1, $$ is equal to $\sum_{k=0}^{n-l}q^{\frac{k(k-1)}{2}}{n-l \choose k}_q$ for fixed $c_1, c_2, \cdots, c_l$. Similarly, the graded character of the space $(y_3 \otimes t^{c_1}) \cdots (y_3 \otimes t^{c_l})w_{\psi}, \quad 0\leq c_1 < \cdots < c_l \leq n-1,$ is equal to $\sum_{l=0}^{n}q^{\frac{l(l-1)}{2}}{n \choose l}_q$. Hence, the result holds. 
\end{proof}
For $\psi \in (\mathfrak{h}[t])^*$ such that it satisfies the condition given in \propref{local_basis}, define a set \begin{equation} \label{b(lambda)-set} \begin{split}
    B(\psi):= & \{(a_1,\cdots,a_j,b_1,\cdots,b_k,c_1,\cdots,c_l): 0\leq c_1 < \cdots < c_l \leq \psi(h_2)-1,\\
    & 0 \leq b_1 < \cdots <b_k \leq \psi(h_2)-l-1,\, 0 \leq a_1 \leq \cdots \leq a_j \leq \psi(h_2)-l-k-j\} \end{split}
\end{equation} 
The cardinality of this set $|B(\psi)| = 4^{\psi(h_2)}$. 

\subsection{}  In \cite{pop}, authors gave a new parametrization of the basis of local Weyl modules for current Lie algebra $\mathfrak{sl}_{n+1}[t]$ through POPs (Partition Overlaid Patterns). We begin by recalling the definition of POPs, followed by their main result on the parametrization of the basis of the local Weyl module in the case of \(\mathfrak{sl}_2[t]\).
 \begin{defn} A partition overlaid pattern (POP for short) consists of an integral Gelfand Tsetlin pattern with bounding sequence $n_1 \geq n_2 $, i.e., a pattern $\begin{array}{ccc}
   & m  &  \\
  n_1 &  &n_2   
 \end{array}$ with  $n_1 \geq m \geq n_2 \geq 0$ and a partition that fits into the rectangle with sides $(n_1-m, m-n_2)$, i.e., a partition that has atmost $n_1-m$ parts and each part has length at most $m-n_2.$     
 \end{defn}
 Note that a partition that fits into the rectangle with sides $(a, b)$ means a partition contained in a rectangular array of square boxes, with $b$ square boxes in each of $a$ rows. For example, a partition fits into rectangle with sides $(2,3)$ means a partition contained in \begin{ytableau}
{}   & &  \\
 & &  
\end{ytableau}, it can be any one of partition among $\phi$, $(1)$, $(1,1)$, $(2)$, $(2,1)$, $(2,2)$, $(3)$, $(3,1)$, $(3,2)$, $(3,3)$.
 
 \begin{q} A partition overlay the pattern $\begin{array}{ccc}
   & 1  &  \\
  3 &  & 0   
 \end{array}$ is contained in the rectangle $(2,1)$. The possible partitions are $\phi, (1), (1,1)$. So $\left(\begin{array}{ccc}
   & 1  &  \\
  3 &  & 0   
 \end{array}; \phi\right) $, $\left(\begin{array}{ccc}
   & 1  &  \\
  3 &  & 0   
 \end{array}; (1)\right)$ and $\left(\begin{array}{ccc}
   & 1  &  \\
  3 &  & 0   
 \end{array}; (1,1)\right)$ are POPs.
 \end{q}

 \medskip 

 Let $\mathfrak{P}= \left( \begin{array}{ccc}
   & m  &  \\
  n_1 &  &n_2   
 \end{array}, \underline{\pi} \right)$ be a POP with bounding sequence $\underline{\lambda}: n_1 \geq n_2$ and $\underline{\pi}$ be a partition that fits into the rectangle with sides $(n_1-m, m-n_2)$. Denote by $\underline{\pi}(k)$ the number of parts of the partition that are equal to $k$ for $1 \leq k \leq m-n_2$ and set $\underline{\pi}(0):= n_1-m-\sum\limits_{k=1}^{m-n_2} \underline{\pi}(k)$. For root $\alpha$ of $\mathfrak{sl}_2$, define \begin{equation}
     \rho_{\mathfrak{P}}^{\alpha}:= (y_{\alpha} \otimes 1)^{\underline{\pi}(0)}. (y_{\alpha} \otimes t)^{\underline{\pi}(1)} . \ldots (y_{\alpha} \otimes t^{m-n_2})^{\underline{\pi}(m-n_2)}.\end{equation}\label{basismonomial} 
$$ v_{\mathfrak{P}}^{\alpha}:= \rho_{\mathfrak{P}}^{\alpha}. w_\lambda$$
    
 
 We recall the basis of local Weyl modules given in \cite{pop} in our notation in the case of $\mathfrak{sl}_{2,\alpha}[t].$
\begin{thm} (See \cite[Theorem 4.3]{pop})  The elements $v_{\mathfrak{P}}^{\alpha}$, as $\mathfrak{P}$ varies over all POPs with bounding sequence $n \geq 0$, forms a basis of $W_{loc}(n \omega)$, local Weyl module of current Lie algebra $\mathfrak{sl}_{2,\alpha}[t]$.    
\end{thm}

\subsection{} For $n \in \Z_+$, $M_{2 \times n}(\mathbb Z_2)=\left\lbrace A = \begin{bmatrix}
    a_{11} & a_{12} & \cdots & a_{1n}\\
    a_{21} & a_{22} & \cdots & a_{2n}
\end{bmatrix}: a_{ij} \in \mathbb Z_2, \forall \,1 \leq i \leq 2, \, 1 \leq j \leq n\right\rbrace$  denotes the matrices of order $2 \times n$ with entries from $\mathbb Z_2$. Denote the sum of all entries of the matrix $A$ by $s(A)$, i.e., $s(A)= \sum\limits_{j=1}^{n} (a_{1j}+a_{2j}).$ We can attach a monomial $x_A \in U(\mathfrak{sl}(1|2)[t])$ to each element $A \in M_{2 \times n}(\mathbb Z_2)$ in the following way:
\begin{equation}
    x_A:= \prod\limits_{j=1}^{n}(x_1 \otimes t^{j-1})^{a_{1j}} \prod\limits_{j=1}^{n}(y_3 \otimes t^{j-1})^{a_{2j}}.\end{equation}\label{supermonomial}

Consider a subset \begin{equation}
    \mathcal{S}_n= \left\lbrace \begin{bmatrix}
    a_{11} & a_{12} & \cdots & a_{1n}\\
    a_{21} & a_{22} & \cdots & a_{2n}
\end{bmatrix} \in M_{2 \times n}(\mathbb Z_2): a_{2j} \in \mathbb Z_2, \, a_{1k} =0, \forall \, k > n- \sum\limits_{j=1}^{n} a_{2j}  \right\rbrace . \end{equation}\label{parameterizing_set_S}

 \medskip

We define the notion of super partition overlaid pattern (super POP) for $\mathfrak{sl}(1|2)[t]$.

 \begin{defn} For $ n \in \Z_+$, a super POP consists of a matrix $A$ from $ \mathcal{S}_n$ and a POP with bounding sequence $n-s(A) \geq 0$, where $s(A)$ is the sum of all the entries of the matrix $A$.     
 \end{defn}
 \begin{q}\label{sup-pop} For $n=4$ and $A = \begin{bmatrix}
     0 & 1 & 0 & 0\\
     0  & 0 & 0 &0
 \end{bmatrix}$, the following are some super POPs with bounding sequence $3 \geq 0$ $$\left( 
 \begin{bmatrix}
     0 & 1 & 0 & 0\\
     0  & 0 & 0 &0
 \end{bmatrix}, \begin{array}{ccc}
      &  1 & \\
      3& & 0
 \end{array}; \phi
 \right),\left( 
 \begin{bmatrix}
     0 & 1 & 0 & 0\\
     0  & 0 & 0 &0
 \end{bmatrix}, \begin{array}{ccc}
      &  1 & \\
      3& & 0
 \end{array}; (1)
 \right), \left( 
 \begin{bmatrix}
     0 & 1 & 0 & 0\\
     0  & 0 & 0 &0
 \end{bmatrix}, \begin{array}{ccc}
      &  1 & \\
      3& & 0
 \end{array}; (1,1)
 \right)$$ \qed 
 \end{q}

 \medskip 

\noindent  Let $ n \in \Z_+$. To each super POP $\mathfrak{P^s}=(A, \mathfrak{P}_A)$ where $A \in \mathcal{S}_n$ and $\mathfrak{P}_A$ is a POP with bounding sequence $n-s(A) \geq 0$, let us associate a monomial $$ v_{\mathfrak{P^s}}:= \rho_{\mathfrak P_A}^{\alpha_2}. x_A \in U(\mathfrak g[t])$$ where $x_A \in U(\mathfrak g[t])$ defined in \eqref{supermonomial} for each element $A \in \mathcal{S}_n$ and $\rho_{\mathfrak P}^{\alpha_2}$ as defined in \eqref{basismonomial} for each POP $\mathfrak{P}$ with bounding sequence $n_1 \geq n_2$.  

\begin{q} The monomials attached to super POPs discussed in Example \ref{sup-pop} are $(y_2 \otimes 1)^2(x_1 \otimes t)$, $(y_2 \otimes 1)(y_2 \otimes t)(x_1 \otimes t)$ and $(y_2 \otimes t)^2(x_1 \otimes t)$ respectively. \qed 
\end{q}
\subsection{} For $\psi \in (\mathfrak{h}[t])^*$ such that $n =\psi(h_2)\in \mathbb{Z}_+$ and $A\in \mathcal{S}_n$, define $GT_{A}^m: = \begin{array}{ccc}
         & m & \\
         n-s(A)& & 0    \end{array}$ where $0\leq m \leq n-s(A)$ and \begin{equation} \label{P(lambda)}
             \mathbb P(\psi) =\left\lbrace  \mathfrak{P^s} = \left(A, GT_{A}^m ; \underline{\pi}\right)|\, A \in \mathcal{S}_n, 
         \underline{\pi} \text{ is a partition contained in rectangle } (n-s(A)-m, m) \right\rbrace.
         \end{equation} of super POPs consists of matrices in $\mathcal{S}_n$ with POPs. The following theorem gives a new parametrization of the basis of the local Weyl module given in \propref{local_basis}.
\begin{thm}\label{POP thm} Let $\psi \in (\mathfrak{h}[t])^\ast$ and assume that there exist $\mu_1,\mu_2,\cdots,\mu_s \in P^+$ and pairwise distinct $z_1,z_2,\cdots,z_s \in \mathbb{C}$ such that $\psi(h\otimes t^p) = z_1^p \mu_1(h) + \cdots+ z_s^p \mu_s(h),\, \forall h \in \mathfrak h, p\geq 0$. Then  the elements $v_{\mathfrak{P^s}}. w_\psi$, as $\mathfrak{P^s}$ varies over all super POPs with $n= \psi(h_2)$, form a basis of $W(\psi).$  
\end{thm}
\begin{proof}
    We will show that these basis elements are scalar multiples of basis elements constructed in \cite{macedo} recalled in \propref{local_basis}. We now establish a bijection between the set $\mathbb{P}(\psi)$ defined in \eqref{P(lambda)} and $B(\psi)$ defined in \eqref{b(lambda)-set}. This has the further property that if the super POP $\mathfrak{P^s}$ maps to the element $(a_1, a_2, \cdots, a_j, b_1, \cdots, b_k, c_1, \cdots, c_\ell)$ of $B(\psi)$, 
    then the monomials $v_{\mathfrak{P^s}}$ and $(y_2 \otimes t^{a_1}) \cdots (y_2 \otimes t^{a_j}) (x_1 \otimes t^{b_1})\cdots (x_1 \otimes t^{b_k}) (y_3 \otimes t^{c_1}) \cdots (y_3 \otimes t^{c_l})$  will be the same.

         To define a map, let $\mathfrak{P^s} \in \mathbb P(\psi)$ be a super POP, with matrix $A= \begin{bmatrix}
    a_{11} & a_{12} & \cdots & a_{1n}\\
    a_{21} & a_{22} & \cdots & a_{2n}
\end{bmatrix}$ in $\mathcal{S}_n$, a GT-pattern $GT_{A}^m$
and a partition overlay $\underline{\pi}$ contained in rectangle  $(n-s(A)-m, m)$. Fix $I=\{1, 2, \cdots, n\}$. Choose $I_\ell=\{i_1,  i_2, \cdots ,i_\ell\} \subset I$ with $i_1 < i_2 < \cdots < i_l$ such that $a_{2 i_1}=a_{2 i_2}= \cdots= a_{2 i_\ell}=1$ and $a_{2s}=0 $ for all $s \in I\setminus I_\ell.$  Thus we have $0 \leq \ell \leq n$ where $I_0= \phi$ an empty set.

Similarly, choose $I_{k,\ell}=\{j_1, j_2 , \cdots,  j_k\} \subset I$ with $j_1 < j_2 < \cdots < j_k $, such that $a_{1 j_1}=a_{1 j_2}= \cdots= a_{1 j_k}=1$ and $a_{1s}=0 $ for all $s \in I\setminus I_{k,\ell}.$ 
Since $a_{1p} =0,\forall p > n-l$, we have $0 \leq k \leq n-\ell $ where $I_{0, \ell}= \phi$ an empty set. Here $s(A)= k+\ell.$ Thus we have $$GT_{A}^{m} = \begin{array}{ccc}
         & m & \\
         n-k-\ell& & 0    \end{array}, \text{ with partition overlay $\underline{\pi}$ contained in a rectangle} (n-k-\ell-m, m).$$ Let $j =n-k-\ell-m$, then $\underline{\pi} $ is a partition contained in rectangle $(j, n-k-\ell-j)$, so it has atmost $j$ parts. Let $\underline{\pi}^{\downarrow}:= (\underline{\pi}_1,\underline{\pi}_2, \cdots, \underline{\pi}_j)$ be the sequence of parts of partition $\underline{\pi}$ arranged in weakly increasing order, padded with the appropriate number of zeros at the beginning so that the length of $\underline{\pi}^{\downarrow}$ is $j$.  Note that $0 \leq \underline{\pi}_1\leq \underline{\pi}_2 \leq  \cdots \leq \underline{\pi}_j \leq n-k-\ell-j$.\\  

Now we have $ I_\ell =\{i_p: 1 \leq p \leq \ell, \quad 1 \leq i_1<i_2 < \cdots < i_\ell \leq n 
 \}, \quad I_{k,\ell} =\{j_q: 1 \leq q \leq k, \quad 1 \leq j_1<j_2 < \cdots < j_k \leq n-\ell \},\quad \underline{\pi}^{\downarrow} =(\underline{\pi}_1,\underline{\pi}_2, \cdots, \underline{\pi}_j)$.
         Take $c_p= i_p-1$ for $1 \leq p \leq \ell$ and $b_q= j_q-1$ for $1\leq q \leq k$ and $a_r= \underline{\pi}_r$ for $1 \leq r \leq j.$ Thus we have associated a tuple $\mathfrak{b}(\mathfrak{P^s})= (a_1, a_2, \cdots, a_j, b_1, \cdots, b_k, c_1, \cdots, c_\ell)$ in $B(\psi)$
         to a super POP $\mathfrak{P^s}.$ We will prove that $\mathfrak{P^s}\mapsto \mathfrak{b}(\mathfrak{P^s})$ is a bijection.\\

         To give the reverse map, consider an element $\hat{b} = (a_1, a_2, \cdots, a_j, b_1, \cdots, b_k, c_1, \cdots, c_\ell)$ in $B(\psi)$, 
         then take $i_p= c_p+1, \forall 1 \leq p \leq \ell$; $j_q= b_q+1, \forall 1 \leq q \leq k$  and $\underline{\pi}_r= a_r, \forall 1 \leq r \leq j.$ Then we obtained the sets $I_\ell =\{i_p: 
   1 \leq p \leq \ell, \quad 1 \leq i_1<i_2 < \cdots < i_\ell \leq n 
 \}$, $
  I_{k,\ell} =\{j_q: 1 \leq q \leq k, \quad 1 \leq j_1<j_2 < \cdots < j_k \leq n-\ell \}$ and $\underline{\pi}^{\downarrow} =(\underline{\pi}_1,\underline{\pi}_2, \cdots, \underline{\pi}_j)$. 
  Take a matrix $A= \begin{bmatrix}
    a_{11} & a_{12} & \cdots & a_{1n}\\
    a_{21} & a_{22} & \cdots & a_{2n}
\end{bmatrix}$ such that $a_{2 i_1}=a_{2 i_2}= \cdots= a_{2 i_\ell}=1$ and $a_{2s}=0, \forall s \in I\setminus I_\ell$ and $a_{1 j_1}=a_{1 j_2}= \cdots= a_{1 j_k}=1$ and $a_{1s}=0, \forall s \in I\setminus I_{k,\ell}.$ Since, $j_k \leq n-l$, we have that $a_{1s} =0, \forall s> n-l$ which implies that $A \in \mathcal{S}_n$. Thus corresponding to the matrix $A$, we have a POP $GT_{A}^{m} = \begin{array}{ccc}
     & m & \\
     n-s(A) & & 0 
\end{array}$ where $s(A) = n+\ell$ and $m=n-k-\ell-j $ with a partition overlay $\underline{\pi}$ which is obtained by removing zeros from $\underline{\pi}^\downarrow$ and by writing non-zeroes terms in the decreasing order. Since the partition $\underline{\pi}$ has atmost $j$ parts and the largest part $\underline{\pi}_j \leq n-k-\ell-j$, we have that the partition $\underline{\pi}$ is contained in the rectangle $(j, n-k-\ell-j).$  Therefore, we obtain a super POP $\mathfrak{p}(\hat{b})=\left(A, GT_{A}^m;
\underline\pi \right)$ from the element $\hat{b}$ and a map $\hat{b}\mapsto \mathfrak{p}(\hat{b})$.\\ 

It is clear that the maps defined above are mutual inverses of each other. So the monomial $v_{\mathfrak{P^s}}$ and $(y_2 \otimes t^{a_1}) \cdots (y_2 \otimes t^{a_j}) (x_1 \otimes t^{b_1})\cdots (x_1 \otimes t^{b_k}) (y_3 \otimes t^{c_1}) \cdots (y_3 \otimes t^{c_\ell})$ are exactly the same. This completes the proof of the theorem.
\end{proof}
Using \thmref{POP thm}, we obtain a new graded character formula for the local Weyl module $W(\psi)$ in terms of the super POP, leading to the following immediate corollary.
\begin{cor}
    Let $\psi \in (\mathfrak{h}[t])^\ast$ and assume that there exist $\mu_1,\mu_2,\cdots,\mu_s \in P^+$ and pairwise distinct $z_1,z_2,\cdots,z_s \in \mathbb{C}$ such that $\psi(h\otimes t^p) = z_1^p \mu_1(h) + \cdots+ z_s^p \mu_s(h),\, \forall h \in \mathfrak h, p\geq 0$. Then $$\ch_{gr} (W(\psi)) = \sum_{\mathfrak{P^s}\in \mathbb{P}(\psi)} u_1^{\mathfrak{P^s}(h_1)}u_2^{\mathfrak{P^s}(h_2)}q^{|\mathfrak{P^s}|}$$ where $\mathfrak{P^s}(h_i) \, (\text{ respectively } |\mathfrak{P^s}|)$ denote the $h_i$-weight (respectively grade) of $v_{\mathfrak{P^s}}w_{\psi}$ and the sum varies over all super POPs $\mathfrak{P^s}$ with $n= \psi(h_2)$. 
\end{cor}

In the following example, we compute the basis of $W(\psi)$ for $\psi(h_2)=2$ using super POPs.
\begin{q} Let $\psi \in \mathfrak h[t]^\ast$ such that $\psi(h_2)=2$. Then 
$$\mathcal{S}_2=\left\{ \begin{bmatrix}
    0 & 0 \\
    0 & 0 
\end{bmatrix}, \begin{bmatrix}
    1 & 0 \\
    0 & 0 
\end{bmatrix}, \begin{bmatrix}
    0 & 1 \\
    0 & 0 
\end{bmatrix}, \begin{bmatrix}
    1 & 1 \\
    0 & 0 
\end{bmatrix},
\begin{bmatrix}
    0 & 0 \\
    1 & 0 
\end{bmatrix},
\begin{bmatrix}
    1 & 0 \\
    1 & 0 
\end{bmatrix}, \begin{bmatrix}
    0 & 0 \\
    0 & 1 
\end{bmatrix},  \begin{bmatrix}
    1 & 0 \\
    0 & 1 
\end{bmatrix},
\begin{bmatrix}
    0 & 0 \\
    1 & 1 
\end{bmatrix} \right\}.$$
Note that $s(A)$, the sum of entries of matrix $A$, for each $A \in \mathcal{S}_2$, can be $0,1,2$. If $s(A)=0$, only such $A=\begin{bmatrix}
    0 & 0 \\
    0 & 0 
\end{bmatrix} \in \mathcal{S}_2$ then POPs associated to bounding sequence $2 \geq 0$ are 
$$ \left( \begin{array}{ccc}
      &  0 & \\
      2& & 0
 \end{array}; \phi
 \right), \left( \begin{array}{ccc}
      &  1 & \\
      2& & 0
 \end{array}; \phi
 \right), \left( \begin{array}{ccc}
      &  1 & \\
      2& & 0
 \end{array}; (1)
 \right), \left( \begin{array}{ccc}
      &  2 & \\
      2& & 0
 \end{array}; \phi
 \right) $$ Thus the associated bases elements are $y_2^2 w_\lambda, \,\,  y_2w_\lambda, \,\, (y_2 \otimes t)w_\lambda,  \text{ and } 1.w_\lambda$ respectively.\\ 
 If $s(A)=1$ then POPs associated to bounding sequence $1 \geq 0$ are $\left( \begin{array}{ccc}
      &  0 & \\
      1& & 0
 \end{array}; \phi
 \right), \left( \begin{array}{ccc}
      &  1 & \\
      1& & 0
 \end{array}; \phi
 \right).$  There are $4$ elements in $\mathcal S_2$ such that $s(A)=1$. If we take $A=\begin{bmatrix}
    1 & 0 \\
    0 & 0 
\end{bmatrix}$ then with two POPs, associated basis elements are $y_2 x_1 w_\lambda$ and $x_1 w_\lambda$ respectively. Similarly other such $A$ correspond to $y_2(x_1 \otimes t)w_\lambda, \,\, (x_1 \otimes t)w_\lambda, \,\, y_2 y_3 w_\lambda, \,\, y_3 w_\lambda, \,\,  y_2(y_3 \otimes t)w_\lambda, \text{ and } (y_3 \otimes t)w_\lambda $.
 If $s(A)=2$ then only POP associated to bounding sequence $0 \geq 0$ is $\left( \begin{array}{ccc}
      &  0 & \\
      0& & 0
 \end{array}; \phi
 \right).$ There are again $4$ choices of matrix $A \in \mathcal{S}_2$ such that $s(A)=2$. 
 If we take $A= \begin{bmatrix}
    1 & 1 \\
    0 & 0 
\end{bmatrix}$ then associated basis monomial is $ x_1(x_1 \otimes t) w_\lambda$. Similarly remaining $3$ elements will correspond to $x_1y_3 w_\lambda, \, \, x_1(y_3 \otimes t) w_\lambda, y_3(y_3 \otimes t) w_\lambda.$
\end{q}

\section{A short exact sequence of CV modules}\label{sec4}
\subsection{} Let us recall the definition of CV modules given in \cite{macedo} for $\mathfrak{sl}(1|2)[t]$. 
\begin{defn}
    Given $\lambda \in P^+$ and a partition $\xi = (\xi_0\geq \xi_1 \geq \cdots \geq \xi_n )$ of $\lambda_2 \in \mathbb{Z}_+$. Denote the pair $(\lambda_1,\xi)$ by $\bxi$ and define CV module $V(\bxi)$ to be the quotient of the local Weyl module $W(\lambda)$ by the extra relations \begin{equation}\label{cv-relations}
        (x_2\otimes t)^s (y_2 \otimes 1)^{r+s}, \forall s>kr+\xi_{k+1}+\cdots + \xi_n, \quad r,s >0, \quad 0\leq k \leq n
    \end{equation}
\end{defn}
Let us denote $(x_2\otimes t)^{(s)} (y_2 \otimes 1)^{(r+s)}$ by $Y_2(r,s)$. Also, $V(\bxi)$ is a cyclic $\mathfrak{sl}(1|2)[t]$-module generated by $v_{\boldsymbol{\xi}}$ which is the image of $w_{\lambda}$ onto $V(\bxi)$. It was proved in \cite{macedo} that if $\xi = (1^{\lambda_2})$, then $V(\bxi)= W(\lambda)$, if $\xi = \lambda_2$, then $V(\bxi) = K(\lambda)^0$ and if $\xi = (0)$, then $V(\bxi) = K(0)^0 = \mathbb{C}$.
\begin{lem}(See \cite[Lemma 4.3]{macedo}) For $\lambda \in P^+ , \xi = (\xi_0 \geq \cdots \geq \xi_n)$ a partition of $\lambda_2 \in \mathbb{Z}_+$ and $\bxi=(\lambda_1,\xi)$. We have that the relations \eqref{cv-relations} of $V(\bxi)$ reduce to $$y_2(r,s)v_{\boldsymbol{\xi}} =0, \forall s> kr+\xi_{k+1}+\cdots +\xi_n, \quad \xi_{k+1}\leq r \leq \xi_k, \quad 0\leq k \leq n.$$
\end{lem}

\noindent

For root $\alpha_2$, we have a copy of $\mathfrak{sl}_{2,\alpha_2}$ in $\mathfrak{sl}(1|2)$. Thus, \cite[Section 2.3, 2.4]{Murray2018} holds in our case. Using similar arguments given in \cite[Theorem 2.2.1]{Murray2018}, we have the following result in our case.
\begin{prop} For $\lambda \in P^+ , \xi = (\xi_0 \geq \cdots \geq \xi_n)$ a partition of $\lambda_2 \in \mathbb{Z}_+$ and $\bxi=(\lambda_1,\xi)$. The module $V(\bxi)$ is quotient of $W(\lambda)$ by the submodule generated by \begin{equation}\label{3rd presentation} \{ y_2(r, |(\xi^{tr})^{(r)}|-r+1)v_{\boldsymbol{\xi}} =0, \, \forall \, 1 \leq r \leq \xi_1-1.\}\end{equation}
where $\xi^{tr}$ is the transpose partition of $\xi$ and $(\xi)^{(r)}:= \sum\limits_{j=1}^{r} \xi_j.$
\end{prop}

\begin{lem}\label{technical lemma} Suppose $V$ is a $\mathfrak{g}[t]$ module  and $v \in V $ such that $$\mathfrak (x_2 \otimes t^r) v=0, \, \forall r\geq 0, \quad (h \otimes t^p)v=0, \forall p \in \mathbb N,\,h\in \mathfrak{h}, \quad  \text{and } Y_2(r,s) v=0, \forall s+r \geq  N $$ for some $N \in \mathbb N$. If $(y_2 \otimes t^{j+1})v=0$ for some $j \in \mathbb N$, then the following holds:
\begin{itemize}
    \item[(i)] $Y_2(r,s)(y_2 \otimes t^j)^k v=0$ for all $r+s \geq N-2k.$
    \item[(ii)] If $(y_3 \otimes t^{b+j})v=0$ for some $b \in \mathbb N$ and $(y_1 \otimes t^r)v =0, \forall r\geq 0$, then $Y_2(r,s)(y_2 \otimes t^j)^k(y_3 \otimes t^b)v=0$ for all $r+s \geq N-2k-1.$
    \item[(iii)] If $(x_1 \otimes t^{a+j})v=0$ for some $a \in \mathbb N$ and $(x_3 \otimes t^r)v =0, \forall r\geq 0$ then $Y_2(r,s)(y_2 \otimes t^j)^k(x_1 \otimes t^a)v=0$ for all $r+s \geq N-2k-1.$
    \item[(iv)] If $(x_1 \otimes t^{a+j}) v=0$, $(y_3 \otimes t^{b+j}) v=0$ for some $a,b \in \mathbb N$ and $(y_1 \otimes t^r)v =0=(x_3 \otimes t^r)v , \forall r\geq 0$ then 
 $Y_2(r,s)(y_2 \otimes t^j)^k(x_1 \otimes t^a)(y_3 \otimes t^b)v=0$ for all $r+s \geq N-2k-2.$
\end{itemize}  
\end{lem}
\begin{proof}(i) There is a copy of $\mathfrak{sl}_{2, \alpha_2}$ in $\mathfrak{sl}(1|2).$ Therefore, part(i) follows from \cite[Corollary 6.6]{cv}.
    (ii) Suppose $(y_3 \otimes t^{b+j}) v =0$ for some $b \in \mathbb{N}$.   Note that $(y_1 \otimes t^b)v=0$ and $[y_1, x_2]=0$,
   \begin{equation}\label{y_1Y_2} \begin{aligned} (y_1 \otimes t^b) Y _2(r,s)(y_2 \otimes t^j)^k v& =   Y_2(r,s)(y_2 \otimes t^j)^k (y_1 \otimes t^b) v+ Y_2(r-1,s)(y_2 \otimes t^j)^k (y_3 \otimes t^b) v\\&+ kY_2(r,s)(y_2 \otimes t^j)^{k-1}(y_3 \otimes t^{b+j})v \end{aligned}\end{equation}
    For $r+s \geq N-2k$, the left-hand side of \eqref{y_1Y_2} is zero by part(i) and so is the right-hand side. Also, the first term of the right-hand side in \eqref{y_1Y_2} is zero as $(y_1 \otimes t^b)v=0$, and the last term of the summation is also zero by our assumption $(y_3 \otimes t^{b+j}).v=0$. Therefore, we have $Y_2(r-1,s)(y_1 \otimes t^j)^k (y_3 \otimes t^b)v=0$ for $r+s \geq N-2k$, i.e., $$Y_2(r,s)(y_1 \otimes t^j)^k (y_3 \otimes t^b)v=0, \text{ for } r+s \geq N-2k-1.$$

    (iii) The proof of this part is similar to the proof given in part (ii). Apply $(x_3 \otimes t^a)$ on $Y_2(r,s)(y_2 \otimes t^j)^k$ and follow the same steps given in part (ii).

    (iv) Assume that $(x_1 \otimes t^{a+j}) v=0$ and $(y_3 \otimes t^{b+j}) v=0$ for some $a,b \in \mathbb N$. We have    
    $$\begin{aligned}
       (x_3 &\otimes t^a)Y_2(r,s)(y_2 \otimes t^j)^k (y_3 \otimes t^b)v \\= &  Y_2(r,s)(y_2 \otimes t^j)^k (x_3 \otimes t^a) (y_3 \otimes t^b) v +Y_2(r-1,s)(y_2 \otimes t^j)^{k}(x_1 \otimes t^a)(y_3 \otimes t^{b})v\\ & + k Y_2(r,s)(y_2 \otimes t^j)^{k-1} (x_1 \otimes t^{a+j}) (y_3 \otimes t^{b})v\\
         =&Y_2(r,s)(y_2 \otimes t^j)^k (h_3 \otimes t^{a+b})v +Y_2(r-1,s)(y_2 \otimes t^j)^{k}(x_1 \otimes t^a)(y_3 \otimes t^{b})v\\ &+kY_2(r,s)(y_2 \otimes t^j)^{k-1}(y_2 \otimes t^{j+a+b})v 
    \end{aligned}$$
   For $ r+s \geq N-2k-1$, $Y_2(r,s)(y_2 \otimes t^j)^k (y_3 \otimes t^b)v=0$. Therefore, the left-hand side is zero, and so is the right-hand side. Also, the first and the third term of the right-hand side is zero by using the assumption hypothesis. Therefore, we have $Y_2(r-1,s)(y_2 \otimes t^j)^k (x_1 \otimes t^a)(y_3 \otimes t^b)v=0$ for $r+s \geq N-2k-1$, i.e., $Y_2(r,s)(y_2 \otimes t^j)^k (x_1 \otimes t^a)(y_3 \otimes t^b)v=0 \text{ for } r+s \geq N-2k-2.$
\end{proof}
\subsection{} Given $\lambda \in P^+$ and a partition $\xi=(\xi_0\geq \xi_1 \geq \cdots \geq \xi_n>0)$ of $\lambda_2$ with $n+1$ parts, i.e., $|\xi|=\lambda_2$. Associate another two partitions $\hat{\xi}, \tilde{\xi}$ to $\xi$ as follows.
$$\hat{\xi} =(\xi_0 \geq \cdots \geq \xi_{n-1} \geq \xi_n-1 \geq 0), \quad 
  \tilde{\xi} =(\xi_0 \geq \cdots \geq \xi_{n-1} \geq \xi_n-2 \geq 0).$$
Note that $\tilde{\xi}$ is defined only if $\xi_n \geq 2.$ If $\xi_n =1$ then we take partition $\tilde{\xi}$ as empty partition. We denote a pair $(\lambda_1,\xi)$ by $\bxi$ and define pairs $\bxi^{\pm}$, $(\hat{\bxi})^-, (\tilde{\bxi})^-$ as follows. If $n=0$ then $\bxi^+=\bxi$ and  $\bxi^- =(\hat{\bxi})^- = (\tilde{\bxi})^- = (\lambda_1, \phi)$ are pairs corresponding to empty partition. If $n>0$ then we take $\bxi^+ = (\lambda_1, \xi^+)$ where $\xi^+=(\xi_0 \geq \xi_1\geq \cdots \geq \xi_\ell+1 \geq \xi_{\ell+1} \geq \cdots \geq \xi_{n-1} \geq \xi_n-1 \geq 0) $ and $\ell$ is the minimal integer such that $\xi_{\ell} = \xi_{n-1} $. Here $|\xi^+|=|\xi|$ and if $\xi^+:=(\xi^+_0 \geq \xi^+_1 \geq \cdots \geq \xi^+_n)$ then
\begin{equation}\label{xi_inequality}
    \sum\limits_{j\geq k+1}\xi_j^+ = \left\lbrace \begin{array}{ll}
   \sum\limits_{j\geq k+1}\xi_j,  &  1 \leq k < \ell \\
 -1+ \sum\limits_{j\geq k+1}\xi_j,  &  k \geq \ell
\end{array} \right.\end{equation}
For $n\in \mathbb Z_+$, Define a function $f_n (\bxi) = (\lambda_1 -n, \xi)$ and \begin{equation}\label{xi equations}
\begin{aligned}
    \xi^- =(\xi_0 \geq \cdots \geq \xi_{n-2} \geq \xi_{n-1}-\xi_n \geq 0), & \quad \bxi^- = (\lambda_1, \xi^-),   
\end{aligned}\end{equation}
Similarly, we can define $(\hat{\xi})^-, (\hat{\bxi})^-, (\tilde{\xi})^-$ and $(\tilde{\bxi})^-$. Note that $|\xi^-|=|\xi|-2\xi_n$, $|(\hat{\xi})^-|=|\xi|-2\xi_n+1$, $|(\tilde{\xi})^-|=|\xi|-2\xi_n+2$. Observe that if $\xi_n=1$ then $\hat{\xi}=(\hat{\xi})^-$. If $\xi_{n-1} = \xi_n$, using \eqref{xi equations} we have $(\hat{\xi})^- = (\xi_0\geq \cdots \geq \xi_{n-2}\geq 1)$ and the operations given above,

\begin{equation}\label{xi^- equations}
\begin{aligned}
    \widehat{(\hat{\xi})^-} =(\xi_0 \geq \cdots \geq \xi_{n-2} \geq 0), & \quad \widehat{(\hat{\bxi})^-} = (\lambda_1, \widehat{(\hat{\xi})^-)},\\
    ((\hat{\xi})^-)^- =(\xi_0 \geq \cdots \geq \xi_{n-2} -1 \geq 0), & \quad ((\hat{\bxi})^-)^- = (\lambda_1, ((\hat{\xi})^-)^-),   \\ 
  ((\hat{\xi})^-)^+ =(\xi_0 \geq \cdots \geq \xi_l +1\geq \xi_{l+1} \geq \cdots \geq\xi_{n-2}\geq 0), & \quad ((\hat{\bxi})^-)^+ = (\lambda_1, ((\hat{\xi})^-)^+)  
\end{aligned}\end{equation}  
where $l$ is the least positive integer such that $\xi_l = \xi_{n-2}$.
\begin{lem}\label{surjective map}
  For $\lambda \in P^+$ and $\xi=(\xi_0 \geq \cdots \geq \xi_n>0)$ a partition of $\lambda_2.$ Let $\bxi = (\lambda_1,\xi)$ and if $n>0$, then there exists a surjective map $\phi : V(\bxi) \rightarrow V(\bxi^+)$ such that $\ker \phi = \bu(\mathfrak{g}[t])(y_2 \otimes t^n)^{\xi_n}v_{\boldsymbol{\xi}}$
\end{lem}
\begin{proof}
    To prove that the map $\phi$ exists, it is enough to show that $$y_2 (r,s) v_{\boldsymbol{\xi}^+}=0, \forall s > kr+ \sum\limits_{j \geq k+1}^{n}\xi_j.$$
    It follows immediately as by \eqref{xi_inequality}, $\sum\limits_{j \geq k+1}^{n}\xi_j \geq \sum\limits_{j \geq k+1}^{n}\xi_j^+$. Using \eqref{3rd presentation}, we have $y_2(\xi_n, n\xi_n)v_{\boldsymbol{\xi}^+}=0$, i.e.,  $$(y_2 \otimes t^n)^{\xi_n}v_{\boldsymbol{\xi}^+}=0, \, i.e., \quad  (y_2 \otimes t^n)^{\xi_n}v_{\boldsymbol{\xi}} \in \ker \phi.$$ Therefore, $U(\mathfrak{g}[t])(y_2 \otimes t^n)^{\xi_n} v_{\boldsymbol{\xi}} \subseteq \ker\phi.$ On comparing the partitions $\xi$ and $\xi^+$, only $\xi_n$ and $\xi_\ell$ parts of $\xi$ are different from $\xi_n^+$ and $\xi_\ell^+$ of $\xi^+$ respectively where $\xi_\ell=\xi_{\ell+1}=\cdots= \xi_{n-1}.$ Thus, when $r<\xi_n$ or $r>\xi_\ell$, $$ y_2(r,s)v_{\boldsymbol{\xi}^+}=0, \quad \text{implies that}\quad y_2(r,s)v_{\boldsymbol{\xi}}=0 $$ because  $s=|((\xi^+)^{tr})^{(r)}|-r+1 = |(\xi^{tr})^{(r)}|-r+1 .$  
We need to consider the case when $\xi_n \leq r \leq \xi_\ell=\xi_{n-1}$, i.e., $y_2(r, |((\xi^+)^{tr})^{(r)}|-r+1)v_{\boldsymbol{\xi}^+} =0$, implies $$\ker \phi= \left\langle y_2(r, |((\xi^+)^{tr})^{(r)}|-r+1)v_{\boldsymbol{\xi}}: \xi_n \leq r \leq \xi_{n-1} \right\rangle.$$ If $r=\xi_n$, then $|((\xi^+)^{tr})^{(r)}|-r+1=n\xi_n.$ If $\xi_n < r \leq \xi_\ell=\xi_{n-1}$ then  $|((\xi^+)^{tr})^{(r)}|-r+1= r(n-1)+\xi_n$, i.e, $y_2(r, |((\xi^+)^{tr})^{(r)}|-r+1)v_{\boldsymbol{\xi}}= y_2(r, r(n-1)+\xi_n)v_{\boldsymbol{\xi}}= (y_2 \otimes t^n)^{\xi_n} (y_2 \otimes t^{n-1})^{r-\xi_n}v_{\boldsymbol{\xi}} \in U(\mathfrak g[t] (y_2 \otimes t^n)^{\xi_n} v_{\boldsymbol{\xi}}.$ Thus $\ker\phi \subseteq U(\mathfrak{g}[t])(y_2 \otimes t^n)^{\xi_n} v_{\boldsymbol{\xi}}$. 
Hence $\ker\phi= U(\mathfrak{g}[t])(y_2 \otimes t^n)^{\xi_n} v_{\boldsymbol{\xi}}.$\\
\end{proof}
\subsection{} Let us do the analysis of the case $n=1$ and $\xi_{0}=\xi_1$. In the rest of the section, the symbol $\cong$ denotes the $\mathfrak{g}[t]$-isomorphism. 
\begin{thm}\label{n=1}
    Let $\lambda \in P^+$ and $\xi=(\xi_0\geq \xi_0>0)$ be a partition of $\lambda_2.$ Let $\bxi = (\lambda_1,\xi)$ and $v_{\boldsymbol{\xi}}$ be the generator of $V(\bxi)$, then there exists a short exact sequence of $\mathfrak{g}[t]$-modules
$$ 0 \longrightarrow \ker\phi \longrightarrow V(\bxi) \overset{\phi}{\longrightarrow} V(\bxi^+) \longrightarrow 0.$$ where $\ker\phi= \bu(\mathfrak{g}[t])(y_2 \otimes t)^{\xi_n}v_{\boldsymbol{\xi}}$. Moreover $\dim V(\bxi) = 4^{2}\xi_0^2$ and  $\ker\phi$ has a filtration\\ $\ker\phi=V_0 \supset V_1 \supset V_2 \supset (0)$
    such that $$\begin{array}{ll}V_0/V_1 \cong \tau_{\xi_0}  V(f_{\xi_0 -1}((\hat{\bxi})^-)),\quad V_1/V_2 \cong \tau_{\xi_0}  V(f_{\xi_0}((\hat{\bxi})^-)), \quad
     V_2/(0) \cong \left\lbrace \begin{array}{ll}
        \tau_{\xi_0} V(f_{\xi_0 -1}((\tilde{\bxi})^-))  &  \text{if } \xi_0 >1\\
        0  & \text{if } \xi_0 =1
     \end{array}\right.\end{array}$$
\end{thm}
\begin{proof}
In $V(\bxi)$, we have the following relations \begin{equation}\label{relationsof V(xi)}
    (y_i \otimes t^2)v_{\boldsymbol{\xi}} = 0 = (x_1\otimes t^2)v_{\boldsymbol{\xi}},\, \text{for } i=2,3.
\end{equation}    Using \lemref{surjective map}, we have a surjective map $\phi: V(\bxi)\rightarrow V(\bxi^+)$ such that $\ker(\phi) = \bu (\mathfrak{g}[t])(y_2 \otimes t)^{\xi_0}v_{\boldsymbol{\xi}} $. Now we will prove that $(y_2 \otimes t)^{\xi_0}v_{\boldsymbol{\xi}}$ lies in the $\bu(\mathfrak{g}[t])$-module generated by the elements $\{(y_2 \otimes t)^{\xi_0 -1}(x_1 \otimes t)v_{\boldsymbol{\xi}}, (y_2 \otimes t)^{\xi_0 -1}(y_3 \otimes t)v_{\boldsymbol{\xi}}\}$.
    Note that $y_2 (\xi_0 +1,\xi_0)v_{\boldsymbol{\xi}} =0$ implies that $(y_2 \otimes 1)(y_2 \otimes t)^{\xi_0}v_{\boldsymbol{\xi}}=0$. Thus $$\begin{array}{rl}
     x_3 y_1 (y_2 \otimes 1)(y_2 \otimes t)^{\xi_0}v_{\boldsymbol{\xi}} & =0  \\
      x_3\left((y_3 \otimes 1)(y_2 \otimes t)^{\xi_0}v_{\boldsymbol{\xi}}+ (y_2 \otimes 1)(y_2 \otimes t)^{\xi_0-1}(y_3 \otimes t)v_{\boldsymbol{\xi}} \right) & =0\\
      c(y_2 \otimes t)^{\xi_0}v_{\boldsymbol{\xi}}+(y_3 \otimes 1)(y_2 \otimes t)^{\xi_0-1}(x_1 \otimes t)v_{\boldsymbol{\xi}}+ x_3(y_2 \otimes 1)(y_2 \otimes t)^{\xi_0-1}(y_3 \otimes t)v_{\boldsymbol{\xi}}& =0
    \end{array}$$
    where $c\in \mathbb C$. Let us consider the first case when $\xi_0>1$. Therefore, we have a filtration of $\ker(\phi) = V_0\supset V_1 \supset V_2 \supset (0)$ where 
    $$\begin{aligned} V_0 &= \bu(\mathfrak{g}[t])(y_2 \otimes t)^{\xi_0 -1}(x_1 \otimes t)v_{\boldsymbol{\xi}}+V_1\\
     V_1 &= \bu(\mathfrak{g}[t])(y_2 \otimes t)^{\xi_0 -1}(y_3 \otimes t)v_{\boldsymbol{\xi}}+V_2\\
     V_2 &=
        \bu(\mathfrak{g}[t])(y_2 \otimes t)^{\xi_0 -2}(x_1 \otimes t)(y_3 \otimes t)v_{\boldsymbol{\xi}}. \end{aligned}$$ 
        Now we will prove that there exist surjective $\mathfrak{g}[t]$-module homomorphisms \\ $\phi_0:\tau_{\xi_0}V(f_{\xi_0-1}((\hat{\bxi})^-))\rightarrow V_0/V_1, \quad \phi_1:\tau_{\xi_0}V(f_{\xi_0}((\hat{\bxi})^-))\rightarrow V_1/V_2, \quad \phi_2:\tau_{\xi_0}V(f_{\xi_0-1}((\tilde{\bxi})^-))\rightarrow V_2/(0)$ such that $$v_0 \mapsto (y_2 \otimes t)^{\xi_0-1}(x_1 \otimes t)v_{\boldsymbol{\xi}}, \, v_1\mapsto (y_2 \otimes t)^{\xi_0-1}(y_3 \otimes t)v_{\boldsymbol{\xi}}, \, v_2\mapsto(y_2 \otimes t)^{\xi_0-1}(x_1 \otimes t)(y_3\otimes t)v_{\boldsymbol{\xi}}$$ where $v_0, v_1,$ and $v_2$ are the generators of CV-modules $\tau_{\xi_0}V(f_{\xi_0-1}((\hat{\bxi})^-)), \tau_{\xi_0}V(f_{\xi_0}((\hat{\bxi})^-))$ and $\tau_{\xi_0}V(f_{\xi_0-1}((\tilde{\bxi})^-))$ respectively. Let us prove the existence of the map $\phi_0$. For $r\geq 0$, $$\begin{array}{ll}
      (y_1 \otimes t^r)(y_2 \otimes t)^{\xi_0-1}(x_1 \otimes t)v_{\boldsymbol{\xi}}& = -(y_2 \otimes t)^{\xi_0-2}(x_1 \otimes t)(y_3 \otimes t^{r+1})v_{\boldsymbol{\xi}} \in V_1\quad (\text{using} \eqref{relationsof V(xi)}),  \\
            (x_2\otimes t^r)(y_2 \otimes t)^{\xi_0-1}(x_1 \otimes t)v_{\boldsymbol{\xi}} & = 0\\
            (x_3 \otimes t^r)(y_2 \otimes t)^{\xi_0-1}(x_1 \otimes t)v_{\boldsymbol{\xi}} & = (y_2 \otimes t)^{\xi_0-2}(x_1 \otimes t^{r+1})(x_1 \otimes t)v_{\boldsymbol{\xi}} =0 \quad (\text{using} \eqref{relationsof V(xi)})
        \end{array}$$
    Therefore, we have $\mathfrak{n}^+[t]v_{\boldsymbol{\xi}}=0$, $(h\otimes t^r)(y_2 \otimes t)^{\xi_0-1}(x_1 \otimes t)v_{\boldsymbol{\xi}} = 0 $ if $r>1$, $(h_1\otimes 1)(y_2 \otimes t)^{\xi_0-1}(x_1 \otimes t)v_{\boldsymbol{\xi}} = (\lambda_1 -\xi_0+1) (y_2 \otimes t)^{\xi_0-1}(x_1 \otimes t)v_{\boldsymbol{\xi}}$ and $(h_2 \otimes 1)(y_2 \otimes t)^{\xi_0-1}(x_1 \otimes t)v_{\boldsymbol{\xi}} = (y_2 \otimes t)^{\xi_0-1}(x_1 \otimes t)v_{\boldsymbol{\xi}}$. This proves that $(y_2 \otimes t)^{\xi_0-1}(x_1 \otimes t)v_{\boldsymbol{\xi}}$ is the highest weight vector of weight $(\lambda_1-\xi_0+1,1)$. To prove the existence of $\phi_0$, it is sufficient to prove that $(y_2 \otimes t)(y_2 \otimes t)^{\xi_0-1}(x_1 \otimes t)v_{\boldsymbol{\xi}} =0$. Now $$(y_2\otimes t)(y_2 \otimes t)^{\xi_0-1}(x_1 \otimes t)v_{\boldsymbol{\xi}} = (y_2 \otimes t)^{\xi_0}(x_1 \otimes t)v_{\boldsymbol{\xi}} = x_3y_{2}(\xi_0+1,\xi_0+1)v_{\boldsymbol{\xi}} =0.$$  Similarly, it is easy to prove the existence of $\phi_1$ and $\phi_2$. We have shown the existence of the map $\phi_i$. To show that each $\phi_i$ is an isomorphism, we refer to \remref{dim_argument} given at the end of this section. 
    
    The proof of the case $\xi_0 =1$ is similar to the proof given above by taking $V_2 =(0)$.
\end{proof}
The following is the main result of this section.
\begin{thm} \label{shortexactsequence} Let  $\lambda \in P^+$ and $\xi=(\xi_0\geq \cdots \geq \xi_{n-1}\geq \xi_n>0)$ be a partition of $\lambda_2$ such that $\xi_{n-1}> \xi_n$ or $\xi_{n-1}=\xi_n\geq 2$. Let $\bxi = (\lambda_1,\xi)$, $v_{\boldsymbol{\xi}}$ be the generator of $V(\bxi)$ and if $n>0$, then there exists a short exact sequence of $\mathfrak{g}[t]$-modules
$$ 0 \longrightarrow \ker\phi \longrightarrow V(\bxi) \overset{\phi}{\longrightarrow} V(\bxi^+) \longrightarrow 0.$$ where $\ker\phi= \bu(\mathfrak{g}[t])(y_2 \otimes t^n)^{\xi_n}v_{\boldsymbol{\xi}}$. Moreover $\dim V(\xi) = 4^{n+1}\xi_0 \xi_1\cdots \xi_n$.
\begin{enumerate}
    \item[(i)] If $\xi_{n-1} > \xi_n$, then $\ker\phi$ has a filtration $\ker\phi=V_0 \supset V_1 \supset V_2 \supset V_3 \supset (0)$
    such that $$\begin{array}{ll}V_0/V_1 \cong \tau_{s} V(f_{\xi_n}(\bxi^-)), & V_1/V_2 \cong \tau_{s} V(f_{\xi_n -1}((\hat{\bxi})^-)),\\
     V_2/V_3 \cong \tau_{s} V(f_{\xi_n}((\hat{\bxi})^-)), \text{  and} &  V_3/(0) \cong \left\lbrace \begin{array}{ll}
        \tau_{s} V(f_{\xi_n -1}((\tilde{\bxi})^-))  &  \text{if } \xi_n >1\\
        0  & \text{if } \xi_n =1
     \end{array}\right.\end{array}.$$

\item[(ii)] If $\xi_{n-1}=\xi_n \geq 2$, $n\geq 2$
then $\ker\phi$ has a filtration $\ker\phi=V_0 \supset V_1 \supset V_2 \supset V_3 \supset V_4 \supset V_5 \supset (0)$ 
  such that $$\begin{array}{ll}V_0/V_1 \cong \tau_{s} V(f_{\xi_n}(\bxi^-)), & V_1/V_2 \cong 
      \tau_{n\xi_n} V(f_{\xi_n -1}(((\hat{\bxi})^-)^+)), \\ 
     V_2/V_3 \cong
         \tau_{s +n-1} V(f_{\xi_n}(((\hat{\bxi})^-)^-)),
     &  V_3/V_4 \cong
         \tau_{s+n-1} V(f_{\xi_n -1}(\widehat{(\hat{\bxi})^-)}), \\
V_4/V_5 \cong \tau_{s} V(f_{\xi_n}((\hat{\bxi})^-)),  & 
     V_5/(0) \cong
        \tau_{s} V(f_{\xi_n -1}((\tilde{\bxi})^-)  
     \end{array}.$$
where $ s = n\xi_n;\,  \bxi^+,\bxi^-, (\hat{\bxi})^-, 
((\hat{\bxi})^-)^+, ((\hat{\bxi})^-)^-, \widehat{(\hat{\bxi})^-}, \text{ and } (\tilde{\bxi})^-$ are defined in \eqref{xi equations} and \eqref{xi^- equations}. 


\end{enumerate} 
\end{thm}
The proof of the above theorem covers the rest of the section. 
\subsection{} \textit{proof of part (i) of \thmref{shortexactsequence}.}\\
Let us consider the case when $\xi_{n-1} > \xi_n >1$. Define $$\begin{array}{ll} V_3  = U(\mathfrak{g}[t])(y_2 \otimes t^n)^{\xi_n-2}(x_1 \otimes t^n)(y_3 \otimes t^n)v_{\boldsymbol{\xi}}, & V_2  = U(\mathfrak{g}[t])(y_2 \otimes t^n)^{\xi_n-1}(y_3 \otimes t^n)v_{\boldsymbol{\xi}}+V_3\\  V_1 =  U(\mathfrak{g}[t])(y_2 \otimes t^n)^{\xi_n-1}(x_1 \otimes t^n)v_{\boldsymbol{\xi}}+V_2, & V_0 = U(\mathfrak{g}[t])(y_2 \otimes t^n)^{\xi_n}v_{\boldsymbol{\xi}}+V_1 \end{array}$$ then $\Ker \phi = V_0 \supset V_1 \supset V_2 \supset V_3 \supset (0)$. We will prove that there exist surjective $\mathfrak{g}[t]$-module homomorphisms   $\phi_0:\tau_{s} V(f_{\xi_n}(\bxi^-)) \rightarrow V_0/V_1$, $\phi_1: \tau_{s} V(f_{\xi_n -1}((\hat{\bxi})^-)) \rightarrow V_1/V_2$, $\phi_2: \tau_{s} V(f_{\xi_n}((\hat{\bxi})^-)) \rightarrow V_2/V_3$ and $\phi_3: \tau_{s} V(f_{\xi_n -1}((\tilde{\bxi})^-))  \rightarrow V_3/(0)$ such that $\phi_0(v_0)= (y_2\otimes t^n)^{\xi_n} v_{\boldsymbol{\xi}}$, $\phi_1(v_1)= (y_2\otimes t^n)^{\xi_n-1}(x_1 \otimes t^n) v_{\boldsymbol{\xi}}$,  $\phi_2(v_2)= (y_2\otimes t^n)^{\xi_n-1}(y_3 \otimes t^n) v_{\boldsymbol{\xi}}$, and $\phi_3(v_3)= (y_2\otimes t^n)^{\xi_n-2}(x_1 \otimes t^n)(y_3 \otimes t^n) v_{\boldsymbol{\xi}}$, where $s = n\xi_n; \, v_0,v_1,v_2$ and $v_3$ are the generators of $\tau_s V(f_{\xi_n}(\bxi^-))$, $\tau_s V(f_{\xi_n -1}((\hat{\bxi})^-))$, $\tau_s V(f_{\xi_n}((\hat{\bxi})^-))$ and $\tau_s V(f_{\xi_n -1}((\tilde{\bxi})^-))$ respectively.

Note that $$\begin{array}{ll}
 ( y_{1}\otimes t^r) (y_2\otimes t^n)^{\xi_n} v_{\boldsymbol{\xi}}  &= (y_2\otimes t^n)^{\xi_n-1} (y_3 \otimes t^{n+r})v_{\boldsymbol{\xi}}=0\mod V_1,\forall r\geq 0 \\
    (x_3 \otimes t^r) (y_2\otimes t^n)^{\xi_n} v_{\boldsymbol{\xi}} &=  (y_2\otimes t^n)^{\xi_n-1} (x_1 \otimes t^{n+r})v_{\boldsymbol{\xi}}=0\mod V_1, \forall r\geq 0\\
    (x_2\otimes t^r)(y_2\otimes t^n)^{\xi_n} v_{\boldsymbol{\xi}} &= c_1  (y_2\otimes t^n)^{\xi_n-2} (y_2 \otimes t^{2n+r})v_{\boldsymbol{\xi}}=0,\, c_1 \in \mathbb Z, \forall r\geq 0\\
    (h \otimes t^r).(y_2\otimes t^n)^{\xi_n} v_{\boldsymbol{\xi}} &= c_2(y_2\otimes t^n)^{\xi_n-1} (y_2 \otimes t^{n+r})v_{\boldsymbol{\xi}} =0, c_2 \in \mathbb Z , \forall \, r \geq 1
\end{array}$$
 Since $\phi^+=\{ -\alpha_1, \alpha_2,\alpha_3\}$, we have $\mathfrak n^+[t] (y_2\otimes t^n)^{\xi_n} v_{\boldsymbol{\xi}}=0$ in $V_0/V_1 $ and $(h_i \otimes 1)(y_2\otimes t^n)^{\xi_n} v_{\boldsymbol{\xi}} = (\lambda_i-i\xi_n)(y_2\otimes t^n)^{\xi_n} v_{\boldsymbol{\xi}}$ for $i=1,2$. Thus $(y_2\otimes t^n)^{\xi_n} v_{\boldsymbol{\xi}}$ is the highest weight vector of weight $(\lambda_1-\xi_n, \lambda_2-2\xi_n).$ To prove the existence of map $\phi_0$, it is enough to show that $$y_2(r,s) (y_2\otimes t^n)^{\xi_n} v_{\boldsymbol{\xi}}=0, \forall s > kr+\sum\limits_{j \geq k+1} \xi_j^-.$$ This follows using the similar arguments given in \cite[Proposition 6.7]{cv}.\\
 Now, we will show the existence of the map $\phi_1$. Using similar arguments given above and \eqref{ysqre equation}, it is easy to check that $\mathfrak{n}^+[t] (y_2\otimes t^n)^{\xi_n-1}(x_1 \otimes t^n) v_{\boldsymbol{\xi}} =0,\quad (h\otimes t^r)(y_2\otimes t^n)^{\xi_n-1}(x_1 \otimes t^n) v_{\boldsymbol{\xi}} =0,\, r\geq 1$ in $V_1/V_2$. Note that $(h_i\otimes 1) (y_2\otimes t^n)^{\xi_n-1}(x_1 \otimes t^n) v_{\boldsymbol{\xi}}= (\lambda_i -i\xi_n +1) (y_2\otimes t^n)^{\xi_n-1}(x_1 \otimes t^n) v_{\boldsymbol{\xi}}
 $ for $i=1,2$. Thus $(y_2 \otimes t^n)^{\xi_n-1}(x_1 \otimes t^n) v_{\boldsymbol{\xi}}$ is the highest weight vector in $V_1/V_2$ of weight $(\lambda_1-\xi_n+1, \lambda_2-2\xi_n+1)$. To prove the existence of the map $\phi_1$, it is enough to show that \begin{equation} \label{Y2x1}
     Y_2(r,s) (y_2\otimes t^n)^{\xi_n-1}(x_1 \otimes t^n) v_{\boldsymbol{\xi}}=0, \forall s > kr+\sum\limits_{j \geq k+1} (\hat{\xi})_j^-.
 \end{equation}
 If $k \leq n-2$, $ kr+\sum\limits_{j \geq k+1} (\hat{\xi})^-_j = kr+ \sum\limits_{j \geq k+1} \xi_j -2\xi_n+1$. Since $Y_2(r,s)v_{\boldsymbol{\xi}}=0$ for $r+s > r+kr+\sum\limits_{j \geq k+1} \xi_j$ and $(x_1 \otimes t^{2n})v_{\boldsymbol{\xi}}=0$, then result follows from \lemref{technical lemma}(iii). If $k = n-1$ then $s > (n-1)r$ which implies that $y_2(r,s)= \sum\limits_{p \geq n} X_p (y_2 \otimes t^p)$ for some $X_p \in U(\mathfrak n^-[t])$ which gives \begin{equation}\label{s.2}
 y_2(r,s)(y_2 \otimes t^n)^{\xi_n-1}(x_1 \otimes t^n)v_{\boldsymbol{\xi}}= \sum\limits_{p \geq n} X_p (y_2 \otimes t^p)(y_2 \otimes t^n)^{\xi_n-1}(x_1 \otimes t^n)v_{\boldsymbol{\xi}}.
 \end{equation} Using \eqref{3rd presentation}, we have $y_2(\xi_n+1, n(\xi_n+1))v_{\boldsymbol{\xi}}=0$, i.e., $(y_2 \otimes t^n)^{\xi_n+1}v_{\boldsymbol{\xi}}=0.$ Thus $(y_2 \otimes t^n)^{\xi_n}(x_1 \otimes t^n)v_{\boldsymbol{\xi}}=0.$ Hence the right-hand side of \eqref{s.2} is zero, so is the left-hand side. 
 Therefore, the map $\phi_1 $ exists and is well defined. Using similar arguments and parts (ii) and (iv) of \lemref{technical lemma}, we can easily verify the existence of map $\phi_2$ and $\phi_3$, respectively. Also, define $V_3 =(0)$ when $\xi_{n-1}>\xi_n =1$ and a similar proof holds. The existence of each $\phi_i$ is established above. To show that each $\phi_i$ is an isomorphism, we refer to \remref{dim_argument} given at the end of this section.
\subsection{} \textit{proof of part (ii) of \thmref{shortexactsequence}.}\\
Let us consider the case when $\xi_{n-1} = \xi_{n} >1$. Define $$\begin{array}{ll} V_5  = U(\mathfrak{g}[t])(y_2 \otimes t^n)^{\xi_n-2}(x_1 \otimes t^n)(y_3 \otimes t^n)v_{\boldsymbol{\xi}}, & V_3 = U(\mathfrak{g}[t])(y_2 \otimes t^n)^{\xi_n-1}(x_1 \otimes t^{n-1})(x_1 \otimes t^n)v_{\boldsymbol{\xi}}+V_4\\ V_4 =  U(\mathfrak{g}[t])(y_2 \otimes t^n)^{\xi_n-1}(y_3 \otimes t^n)v_{\boldsymbol{\xi}}+V_5,   & V_2 = U(\mathfrak{g}[t])(y_2 \otimes t^{n-1})(y_2 \otimes t^n)^{\xi_n-1}(x_1 \otimes t^n)v_{\boldsymbol{\xi}}+V_3,\\  V_1 = U(\mathfrak{g}[t])(y_2 \otimes t^n)^{\xi_n-1}(x_1 \otimes t^n)v_{\boldsymbol{\xi}}+V_2  & V_0 = U(\mathfrak{g}[t])(y_2 \otimes t^n)^{\xi_n}v_{\boldsymbol{\xi}}+V_1 \end{array}$$ 
We will prove that there exist surjective $\mathfrak g[t]$-module homomorphisms when $n\geq 2$. $$\begin{array}{ll}
\phi_0: \tau_{s} V(f_{\xi_n}(\bxi^-)) \longrightarrow V_0/V_1, &\quad \phi_1:  \tau_{s} V(f_{\xi_n -1}(((\hat{\bxi})^-)^+)) \longrightarrow V_1/V_2,\\ 
\phi_2: \tau_{s +n-1} V(f_{\xi_n}(((\hat{\bxi})^-)^-)) \longrightarrow V_2/V_3, &\quad \phi_3: \tau_{s+n-1} V(f_{\xi_n -1}(\widehat{(\hat{\bxi})^-)}) \longrightarrow V_3/V_4, \\
\phi_4: \tau_{s} V(f_{\xi_n}((\hat{\bxi})^-))\longrightarrow V_4/V_5, &\quad \phi_5:  \tau_{s} V(f_{\xi_n -1}((\tilde{\bxi})^-) \longrightarrow V_5/(0)
\end{array}$$
such that $\phi_0(v_0)= (y_2\otimes t^n)^{\xi_n} v_{\boldsymbol{\xi}}$, $\phi_1(v_1)= (y_2\otimes t^n)^{\xi_n-1}(x_1 \otimes t^n) v_{\boldsymbol{\xi}}$,  $\phi_2(v_2)= (y_2 \otimes t^{n-1})(y_2\otimes t^n)^{\xi_n-1}(x_1 \otimes t^n) v_{\boldsymbol{\xi}}$, $\phi_3(v_3)= (y_2\otimes t^n)^{\xi_n-1}(x_1 \otimes t^{n-1})(x_1 \otimes t^n) v_{\boldsymbol{\xi}}$, $\phi_4(v_4)=(y_2\otimes t^n)^{\xi_n-1}(y_3 \otimes t^n) v_{\boldsymbol{\xi}}$ and $\phi_5(v_5)= (y_2\otimes t^n)^{\xi_n-2}(x_1 \otimes t^n)(y_3 \otimes t^n) v_{\boldsymbol{\xi}}$, where $v_0, v_1, v_2, v_3,v_4$ and $v_5$ denote the generators of $\tau_s V(f_{\xi_n}(\bxi^-)), \tau_s V(f_{\xi_n -1}(((\hat{\bxi})^-)^+)), \tau_{s+n-1}V(f_{\xi_n}(((\hat{\bxi})^-)^-)),\tau_{s+n-1}V(f_{\xi_n -1}(\widehat{(\hat{\bxi})^-)}),$ $ \tau_s V(f_{\xi_n }((\hat{\bxi})^-))$ and $\tau_s V(f_{\xi_n -1}((\tilde{\bxi})^-)$ respectively.\\
The existence of maps $\phi_0, \phi_4, \phi_5$ has already been proved in part(i). Here, we will prove the existence of $\phi_1, \phi_2, \text{ and } \phi_3$. Let us begin with $\phi_1$. 
Using similar calculation given in the proof of part(i), it is easy to prove that $\mathfrak n^+[t]. (y_2\otimes t^n)^{\xi_n-1}(x_1 \otimes t^n) v_{\boldsymbol{\xi}}=0$ in $V_2/V_3$, i.e., $(y_2\otimes t^n)^{\xi_n-1}(x_1 \otimes t^n) v_{\boldsymbol{\xi}}$ is the highest weight vector of weight $(\lambda_1-\xi_n+1, \lambda_2-2 \xi_n+1).$ To prove the existence of $\phi_1$, we will prove that the following equation holds  \begin{equation}\label{y2y2x1}
     Y_2(r,s)(y_2\otimes t^n)^{\xi_n-1}(x_1 \otimes t^n) v_{\boldsymbol{\xi}} =0, \text{ for all } s > kr+ \sum\limits_{j \geq k+1} ((\hat{\xi})^-)^+_j.
 \end{equation}
Using \lemref{surjective map}, we have $\dfrac{V(f_{\xi_n -1}((\hat{\bxi})^-))}{U(\mathfrak g[t]) (y_2 \otimes t^{n-1})} \cong V(f_{\xi_n -3}(((\hat{\bxi})^-)^+)) $. Hence it is enough to show that
 \begin{align}
   Y_2(r,s)(y_2\otimes t^n)^{\xi_n-1}(x_1 \otimes t^n) v_{\boldsymbol{\xi}}   & =0, \text{ for all } s > kr+ \sum\limits_{j \geq k+1} (\hat{\xi})^-_j \label{s.6}\\
    (y_2 \otimes t^{n-1})(y_2\otimes t^n)^{\xi_n-1}(x_1 \otimes t^n) v_{\boldsymbol{\xi}} & =0 \mod V_2, \label{s.7}  
 \end{align}
Equation \eqref{s.6} is same as equation \eqref{Y2x1} and the proof of \eqref{Y2x1} has already given in part(i). The proof of \eqref{s.7} is trivial. From \eqref{s.6} and \eqref{s.7}, we have \eqref{y2y2x1} holds.
Therefore, the map $\phi_1 $ exists and is well defined. \\
Now for the existence of $\phi_2$, note that $y_2 (\xi_n +1,n\xi_n +n-1)v_{\boldsymbol{\xi}} =0$ implies that $(y_2 \otimes t^n)^{\xi_n}(y_2 \otimes t^{n-1})v_{\boldsymbol{\xi}}=0$. Thus $$\begin{array}{rl}
     x_3 y_1 (y_2 \otimes t^n)^{\xi_n}(y_2 \otimes t^{n-1})v_{\boldsymbol{\xi}} & =0  \\
      x_3\left((y_2 \otimes t^n)^{\xi_n}(y_3 \otimes t^{n-1})v_{\boldsymbol{\xi}}+ (y_2 \otimes t^n)^{\xi_n-1}(y_2 \otimes t^{n-1})(y_3 \otimes t^n)v_{\boldsymbol{\xi}} \right) & =0\\
      (y_2 \otimes t^n)^{\xi_n-1}(x_1 \otimes t^{n})(y_3 \otimes t^{n-1})v_{\boldsymbol{\xi}}+ (x_1 \otimes t^{n-1})(y_2 \otimes t^n)^{\xi_n-1}(y_3 \otimes t^{n})v_{\boldsymbol{\xi}} \\
     +  (y_2 \otimes t^{n-1})(y_2 \otimes t^n)^{\xi_n-2}(x_1 \otimes t^{n}) (y_3\otimes t^{n})v_{\boldsymbol{\xi}}& =0
    \end{array}$$
The second and third term of the above equation are already in $V_3$. Hence $(y_2 \otimes t^n)^{\xi_n-1}(x_1 \otimes t^{n})(y_3 \otimes t^{n-1})v_{\boldsymbol{\xi}}  \in  V_3$ which proves that $\mathfrak n^+[t]. (y_2 \otimes t^{n-1})(y_2\otimes t^n)^{\xi_n-1}(x_1 \otimes t^n) v_{\boldsymbol{\xi}}=0$ in $V_2/V_3$.
Thus $(y_2 \otimes t^{n-1})(y_2\otimes t^n)^{\xi_n-1}(x_1 \otimes t^n) v_{\boldsymbol{\xi}}$ is the highest weight vector of weight $(\lambda_1-\xi_n, \lambda_2-2\xi_n-1).$ To prove the existence of $\phi_3$, it is enough to show that $$Y_2(r,s)(y_2 \otimes t^{n-1})(y_2\otimes t^n)^{\xi_n-1}(x_1 \otimes t^n) v_{\boldsymbol{\xi}} =0, \text{ for all } s > kr+ \sum\limits_{j \geq k+1} ((\hat{\xi})^-)^-_j.$$
Since $((\hat{\xi})^-)^-= (\xi_0 \geq \xi_1 \geq \cdots \geq \xi_{n-3} \geq \xi_{n-2}-1) $ 
as $\xi_{n-1}=\xi_n$. If $k \leq n-3$, then $\sum\limits_{j \geq k+1} ((\hat{\xi})^-)^-_j= \sum\limits_{j \geq k+1} \xi_j- 2 \xi_n -1$. Using \eqref{s.6}, we have $Y_2(r,s)(y_2 \otimes t^n)^{\xi_n-1}(x_1 \otimes t^n) v_{\boldsymbol{\xi}}=0$ for $r+s > r+ kr+ \sum\limits_{j \geq k+1} \xi_j-2\xi_n+1$. Note that $(y_2 \otimes t^n)^{\xi_n-1}(x_1 \otimes t^n) v_{\boldsymbol{\xi}}$ satisfies all the conditions of $v$ in \lemref{technical lemma} for $N=r+ kr+ \sum\limits_{j \geq k+1} \xi_j-2\xi_n+1$,
using part(i) of \lemref{technical lemma}, we get $$Y_2(r,s)(y_2 \otimes t^{n-1})(y_2\otimes t^n)^{\xi_n-1}(x_1 \otimes t^n) v_{\boldsymbol{\xi}} =0, \text{for all } s > kr+ \sum\limits_{j \geq k+1} \xi_j-2 \xi_n-1.$$ 
If $k \geq n-2$ then $s> (n-2)r$, i.e., $y_2(r,s)= \sum\limits_{p \geq  n-1} X_p (y_2 \otimes t^p)$ where $X_p \in U(\mathfrak g^{-\alpha_2}[t])$. Therefore, \begin{equation}\label{s.4} y_2(r,s)(y_2 \otimes t^{n-1})(y_2\otimes t^n)^{\xi_n-1}(x_1 \otimes t^n) v_{\boldsymbol{\xi}}= \sum\limits_{p \geq  n-1} X_p (y_2 \otimes t^p)(y_2 \otimes t^{n-1})(y_2\otimes t^n)^{\xi_n-1}(x_1 \otimes t^n) v_{\boldsymbol{\xi}}\end{equation} 
Using \eqref{3rd presentation}, we have $y_2(\xi_n+1, n\xi_n+n-1)v_{\boldsymbol{\xi}}=0$, i.e., $(y_2 \otimes t^{n-1})(y_2 \otimes t^n)^{\xi_n}v_{\boldsymbol{\xi}}=0.$ Thus $(y_2 \otimes t^{n-1})(y_2 \otimes t^n)^{\xi_n-1}(x_1 \otimes t^n)v_{\boldsymbol{\xi}}+ (x_1 \otimes t^{n-1})(y_2 \otimes t^n)^{\xi_n}v_{\boldsymbol{\xi}}=0$ implies that $(y_2 \otimes t^{n-1})^2(y_2 \otimes t^n)^{\xi_n-1}(x_1 \otimes t^n)v_{\boldsymbol{\xi}}=0.$ Hence the right-hand side of \eqref{s.4} is zero, so is the left-hand side.
Therefore, the map $\phi_3 $ exists and is well-defined. 
 Using similar arguments, it is easy to prove that the map $\phi_4$ exists and is well-defined.
\subsection{} The following proposition is given in \cite{macedo}.
\begin{prop} [See {\cite[Proposition 4.4]{macedo}}]
    Let $\lambda \in P^+$, $\xi = (\xi_0, \cdots, \xi_n)$ a partition of $\lambda_2$ and $\bxi = (\lambda_1,\xi)$. For every choice of pairwise distinct complex numbers $z_0, \cdots, z_n$ and complex numbers $a_0, \cdots,a_n$ satisfying $a_0+\cdots+a_n = \lambda_1$, there exist a surjective $\mathfrak{g}[t]$-module homomorphism $$V(\bxi)\twoheadrightarrow K(a_0, \xi_0)^{z_0}\ast \cdots \ast K(a_n,\xi_n)^{z_n}.$$\qed
\end{prop}
The above proposition gives us \begin{equation}\label{greater}\dim V(\bxi)\geq \prod_{k=0}^n \dim K(a_k,\xi_k)^{z_k}= \prod_{k=0}^n 4\xi_k = 4^{n+1}\xi_0 \cdots \xi_n \end{equation}
Using \thmref{shortexactsequence}, we have $\dim V(\bxi) \leq \dim \Ker(\phi)+\dim V(\bxi^+)$. Now we  claim that 
\begin{equation} \label{lessthan}
    \dim \Ker(\phi)+\dim V(\bxi^+)\leq 4^{n+1}\xi_0\cdots \xi_n.
\end{equation} To prove the claim, we will apply induction on $n$. For $n=1$ we have the following cases. \\
Case I: If $\xi_0> \xi_1 =1$, then using the proof of part(i) of \thmref{shortexactsequence} and \propref{dimension of kac} we have 
we have $$\begin{aligned}
 &\dim \ker(\phi)+\dim V(\bxi^+) \\
 & \leq \dim(V(\bxi^+))+ \dim (\tau_{1}V(f_{1}(\bxi^-)))+ \dim (\tau_{1}V(((\hat{\bxi})^-)))+ \dim (\tau_{1}V(f_{1}((\hat{\bxi})^-))) \\
& = 4(\xi_0+1)+4(\xi_0-1)+4\xi_0 +4\xi_0= 4^2\xi_0
\end{aligned}$$
Case II: If $\xi_0 = \xi_1= 1$, then using the proof of \thmref{n=1} and \propref{dimension of kac} it is easy to see that $\dim\ker(\phi)+\dim V(\bxi^+)\leq \dim(V(\bxi^+))\dim(\tau_{1}  V((\hat{\bxi})^-))+\dim(\tau_{1}  V(f_{1}((\hat{\bxi})^-))) = 4.2+4+4=  4^2$.\\
Apply induction on $\xi_1$, if $\xi_1 =1$ we are done by the arguments given above. Let the result hold for all $p<\xi_1$. Consider the case $\xi_{0}> \xi_1\geq 2$, using \thmref{shortexactsequence}, induction arguments and similar calculations given above we have $$\dim \Ker(\phi)+\dim V(\bxi^+)\leq 4(\xi_0 -\xi_1)+8(\xi_0 -\xi_1+1)+ 4(\xi_0-\xi_1+2)+ 4^2 (\xi_0+1)(\xi_1 -1) = 4^2\xi_0\xi_1.$$ Also, when $\xi_{0} = \xi_1\geq 2$ then $$\dim \ker(\phi)+\dim V(\bxi^+)\leq 2. 4+4.2+4^2(\xi_0+1)(\xi_0 -1) = 4^2 \xi_0^2.$$
Hence we have proved the result for $n=1$ case. Now we only consider the partitions $\xi = (\xi_0,\cdots, \xi_n)$ where $\xi_{n-1}>\xi_n$ or $\xi_{n-1}=\xi_n\geq 2$. Assume that the result holds for all such partitions where the number of parts is strictly less than $n$. We need to prove that the result holds for $n$. Apply induction on $\xi_n$. If $\xi_n =1$, then we are in the case where $\xi_{n-1}>\xi_n$. Using \thmref{shortexactsequence} and induction, it is easy to prove that 
\begin{equation*}
    \dim \Ker(\phi)+\dim V(\bxi^+)\leq 4^{n+1}\xi_0 \cdots \xi_{n-1}.
\end{equation*}
Let the result hold for all $p<\xi_n$. Using the proof of \thmref{shortexactsequence} and induction argument it is easy to prove that \begin{equation*}
    \dim \Ker(\phi)+\dim V(\bxi^+)\leq 4^{n+1}\xi_0 \cdots \xi_n.
\end{equation*}
Using \eqref{greater} and \eqref{lessthan} we have that \begin{equation}\label{equality}
\dim V(\bxi) = 4^{n+1}\xi_0 \cdots \xi_n\end{equation}
\begin{rem}\label{dim_argument}The dimension argument given above proves that each $\phi_i$ given in the proof of \thmref{n=1} is an isomorphism. \end{rem}
\begin{proof}[Proof of \thmref{shortexactsequence}]
    Using \lemref{surjective map}, proof of part(i) - (ii) of \thmref{shortexactsequence},  and \eqref{equality} we have that all surjective $\mathfrak{g}[t]$-module homomorphisms given in the proof of part(i) and part(ii) of \thmref{shortexactsequence} are isomorphisms. Hence the \thmref{shortexactsequence}.
\end{proof}
The following result was already proved in \cite{macedo} but we are providing a new proof as it is a direct implication of \thmref{shortexactsequence}.
\begin{cor}
Given $\lambda = (\lambda_1,\lambda_2) \in P^+$ and $\xi = (\xi_0, \cdots,\xi_n)$ a partition of $\lambda_2$ where $\xi$ satisfies the hypothesis given in \thmref{n=1} or \thmref{shortexactsequence}. Consider $\bxi = (\lambda_1,\xi)$ then for every choice of pairwise-distinct complex numbers $z_0, \cdots, z_n$ and $a_0, \cdots, a_n$ satisfying $a_0+\cdots+a_n = \lambda_1$, there exist an isomorphism of $\mathfrak{g}[t]$-modules $$V(\bxi)\cong_{\mathfrak{g}[t]} K(a_0,\xi_0)^{z_0}\ast \cdots\ast K(a_n,\xi_n)^{z_n}.$$     
\end{cor}
\begin{rem}
    Using \cite[Proposition 6.11]{cv} and \thmref{shortexactsequence}, we can also derive the basis and character formula of Chari-Venkatesh module $V(\bxi)$ given in \cite{macedo}.
\end{rem}
\section*{Acknowledgement} {The first author would like to thank R.Venkatesh for useful suggestions and also acknowledges the postdoctoral research grant from the National Board of Higher Mathematics, reference number 0204/33/2024/R\&D-II/14345.}

\bibliographystyle{plain}
\bibliography{Bibliography}
\end{document}